\documentclass[12pt]{article}
\usepackage{amsmath,amsfonts,amssymb,amsthm,mathtools,bbm,nicefrac}
\usepackage{xcolor}
\usepackage{graphicx,psfrag,epsf}
\usepackage{enumerate,enumitem}
\usepackage[round]{natbib}
\usepackage{url} 
\usepackage{algorithm}
\usepackage{algpseudocode}
\usepackage{booktabs}
\usepackage{cleveref}

\usepackage{multirow,subcaption}

\newif\ifshowcomments
\showcommentsfalse

\ifshowcomments
  \newcommand{\lom}[1]{{\textcolor{red}{\small [lom]: #1}}}

  \newcommand{\lomtext}[1]{{\textcolor{orange}{#1}}}
\else
  \newcommand{\lom}[1]{}

  \newcommand{\lomtext}[1]{#1}
\fi

\newcommand{\alt}{{}^\sharp\!}
\newcommand{\alttwo}{{}^\flat\!}

\DeclareRobustCommand{\bb}[1]{\mathbb{#1}}

\DeclareRobustCommand{\t}[1]{\text{#1}}

\DeclareRobustCommand{\th}{\t{th}}

\DeclareRobustCommand{\set}[1]{\left\{#1 \right\}}
\DeclareRobustCommand{\Set}[2][]{\left\{#1 \, \middle|\, #2 \right\}}

\DeclareRobustCommand{\D}{\Delta}

\DeclareRobustCommand{\vp}{\varphi}

\newif\ifshownew
\shownewtrue
\ifshownew
  
\else
  
\fi

\newif\ifshort
\shorttrue
\newtheorem{definition}{Definition}
\newtheorem{assumption}{Assumption}
\newtheorem{theorem}{Theorem}
\newtheorem{proposition}{Proposition}
\newtheorem{corollary}{Corollary}

\newtheorem{lemma}{Lemma}

\newtheorem{remark}{Remark}

\usepackage[T1]{fontenc}
\usepackage{tgtermes}

\DeclareRobustCommand{\thetitle}{Demonstration Experiments}

\newcommand{\paperauthors}{%
  Guido Imbens\thanks{Graduate School of Business, Stanford University; Amazon.com},
  Lorenzo Masoero\thanks{Amazon.com},
  Alexander Rakhlin\thanks{Department of Brain and Cognitive Sciences and
Statistics and Data Science Center, MIT; Amazon.com}, \\
  Thomas S. Richardson\thanks{Department of Statistics, University of Washington; Amazon.com}, and
  Suhas Vijaykumar\thanks{Department of Economics, UC San Diego; Amazon.com}}

\def\spacingset#1{\renewcommand{\baselinestretch}%
{#1}\small\normalsize}

\newcommand{\blind}{1}

\date{}

\begin{document}
\spacingset{1}


\if1\blind
{
  \title{\bf \thetitle}
  \author{\paperauthors}
  \date{}
  \maketitle
} \fi
\if0\blind
{
  \bigskip
  \bigskip
  \bigskip
  \begin{center}
    {\LARGE\bf \thetitle}
\end{center}
  \medskip
} \fi

\bigskip
\begin{abstract}
Adaptive experiments are used extensively in online platforms, healthcare and biotechnology, and the social sciences.
Often, the primary goal is not to precisely estimate a treatment effect but to {demonstrate} that at least one candidate intervention yields a positive effect, for some subpopulation and on some measured outcome. We formalize this objective as testing the global null in a threshold bandit framework, and develop two inference procedures that are valid under general adaptive sampling: one that pools information across promising arms, and one based on time-uniform multiple testing of individual arm means. To support the latter, we establish a moderate-deviations principle for the sequential $t$-statistic, justifying asymptotic confidence sequences in settings where the number of arms is large relative to the sample size. To illustrate how adaptive designs can target the proposed statistics, we recast experimental design as bandit optimization with an arm's reward given by its signal-to-noise ratio, and analyze an allocation rule for which we establish a logarithmic regret bound. We apply the methods in a simulation study of targeting unconditional cash transfer programs.
\end{abstract}
\noindent%



\spacingset{1.8}

\section{Introduction}
\label{sec:intro}

Not all randomized experiments are conducted with the same goal in mind.
In many settings, the primary goal is to obtain a precise estimate of the average effect of an intervention. In others, researchers seek to understand treatment effect heterogeneity, for example by modeling effects as a function of pre-treatment covariates.
A third common objective is to determine an optimal treatment policy from a set of  interventions and targeting rules. 

In this paper, we focus on a fourth objective: determining whether some treatment level has a positive effect on at least one outcome for at least some subpopulation. This question arises naturally in early-stage or exploratory experimentation, where a decision maker must determine whether an intervention shows sufficient promise to justify a larger, more costly, and more tightly controlled experiment. Often, decision makers face a large number of candidate interventions, outcomes, or subpopulations, and must allocate limited experimental resources accordingly.

 We study experimental designs that adaptively allocate samples in order to improve the power to detect the existence of effects which exceed a threshold value. We refer to such studies as \emph{demonstration experiments}, reflecting the idea that their main purpose is to establish the presence of an effect rather than to estimate its magnitude. Such experiments may also be called ``proof-of-concept,'' ``exploratory,'' or ``litmus-test'' experiments; in the sequential decision-making literature, the objective is closely related to the \emph{threshold-bandit} problem \citep{locatelli2016optimal,kaufmann2018sequential}. Unlike classical A/B tests, the goal is not to estimate an average treatment effect precisely, nor to identify the single best arm, but rather to accumulate statistical evidence that at least one meaningful effect is present.

What is the potential benefit of adaptivity in this setting? In the simplest case of a single treatment and a single outcome in a homogeneous population, adaptive designs offer little advantage. However, when treatment effects are heterogeneous across arms, outcomes, or subpopulations, adaptively reallocating samples toward promising effects can substantially improve power relative to uniform designs.

We consider three stylized settings that motivate our framework. First, consider an experiment with a treatment taking $k$ discrete values. As units arrive, the experimenter assigns treatments sequentially, with the aim of testing whether any treatment level exceeds a given benchmark. This setting fits most directly into our formal framework and serves as the primary motivating case throughout the paper.

Second, consider a sequential experiment with a binary treatment and a discrete pre-treatment variable defining mutually exclusive subpopulations. At each stage, the experimenter selects which subpopulation to sample and assigns treatment within that group. The objective is to detect whether the treatment has a positive effect for at least one subpopulation. This setting can be naturally viewed as a multi-armed bandit problem, where each arm corresponds to a subpopulation and the goal is to detect the existence of a positive effect rather than to maximize cumulative reward.

Finally, suppose there are $k$ outcomes of interest, and the goal is to detect a difference between treated and control units for some outcome. Different outcome contrasts may yield different power, and adaptively focusing on informative contrasts can improve detection in settings where measurement is costly. Unlike standard bandit problems, experiments of this form may reveal values for multiple contrasts simultaneously, so that the feedback structure is richer than the classical bandit setting in which only the reward of the chosen action is observed.\footnote{The same is true for inference on subpopulations, if the subpopulations are not mutually exclusive.}

Although these examples differ superficially, they share a common structure: the experimenter repeatedly chooses among several options—treatments, subpopulations, or contrasts—and observes noisy outcomes, with the goal of detecting the existence of a meaningful effect. Our analysis formalizes this structure using a multi-armed bandit framework and develops inference procedures that remain valid under general adaptive sampling.

\subsection{Contributions and related work}

This paper makes three main contributions. First, we formalize the problem of demonstration experiments as a hypothesis-testing problem under adaptive experimental design, where the objective is to detect the existence of a large enough effect across multiple arms rather than to estimate effects precisely or to identify a single best arm. We show that this objective leads naturally to test statistics that are robust to strategic sample allocation.

Second, we propose two testing procedures—pooled testing and max testing—that remain valid under a wide class of adaptive sampling schemes. These statistics capture complementary regimes: the pooled statistic aggregates evidence across arms, while the max statistic focuses on the most promising arm. We establish finite-sample and asymptotic validity guarantees under minimal assumptions on the sampling strategy, accommodating both strategic termination and regimes in which the number of arms grows with the horizon.

Third, we study how experimental designs can be adapted to maximize the power of these tests under alternative hypotheses. We show that the resulting design problem can be viewed as an online optimization problem with bandit feedback, and we propose a simple sampling algorithm that achieves good performance in terms of statistical power. Simulation results illustrate how adaptive allocation can substantially outperform uniform designs, particularly in settings with heterogeneous effects.

\subsubsection{Anytime-Valid Inference and Game-Theoretic Statistics}

Our work draws inspiration from the rapidly growing literature on game-theoretic statistics and anytime-valid inference \citep{ramdas2023game}. This line of research studies tests that remain valid under arbitrary, data-dependent stopping or sampling rules: by constructing tests from statistics which are supermartingales under the null, one may ensure uniform control of type-I error across large classes of experimental designs. Early contributions focused on robustness to optional stopping, and the approach has since been shown to have implications for a wide range of statistical problems. Notable recent work on best-arm identification develops inference methods that remain valid under both strategic termination and adaptive sample allocation \citep{howard2022sequential,liang2023experimental}. 

We, too, consider designs in which the experimenter adaptively chooses whether to continue the experiment and which arm to sample next, though we depart from the best-arm hypotheses studied in prior work; our tests are valid for any sequential allocation satisfying a minimal initialization requirement. Another important distinction is that we study test statistics intentionally optimized by the sampling algorithm: we introduce bandit strategies which aim to maximize the value of the proposed statistics themselves, thereby directly optimizing power. This is conceptually related to ``testing by betting'' \citep{shafer2021testing,shekhar2023nonparametric} and to the results of \citet{casgrain2024sequential}; these works also study online optimization of power for sequential tests, although their full-information setting differs from the bandit feedback we consider.

Most closely related is concurrent and independent work by \citet{sandoval2026multi}, who study a similar many-armed global testing problem using e-process techniques (see also \citealp{hsu2025active}, who study nonparametric two-sample testing across heterogeneous sources). Their framework delivers exact, non-asymptotic validity, but requires constructing wealth processes whose form must be specified to match the distribution of each arm (typically via exponential moments). Our approach is complementary, aimed at standard experimental practice: working within the asymptotic CLT-based tradition, we obtain universal tests and procedures whose form is invariant to the distributions of the arms. Another distinction is that  the aforementioned work develops and realizes an asymptotic notion of optimality---multi-armed log-optimality---which characterizes the optimal asymptotic accumulation rate of evidence against the null. In contrast, our proposed allocation rule targets signal-to-noise in a finite experiment, and achieves a logarithmic finite-sample regret bound.

\subsubsection{Statistical Inference for Multi-Armed Bandits}

 A parallel literature studies valid inference after adaptive allocation in stochastic bandits, often seeking to balance regret minimization with inferential guarantees \citep{simchilevi2023multi,hadad2021confidence,chen2021statistical,bibaut2021post,liang2023experimental,howard2022sequential,chen2023optimal}.\footnote{For a comprehensive overview of inference after adaptive experimentation, we refer the reader to \citet{bibaut2025demystifying}.} Our approach differs in two respects: we recast optimal experimental design itself as an online bandit problem whose objective is signal-to-noise rather than mean response, and we focus on a threshold-crossing null. Taken together, our results suggest that threshold-type hypotheses are substantially easier to test than best-arm hypotheses, in the sense that nearly unconstrained adaptive sampling is compatible with sharp inference.

Beyond these connections, our work also contributes to the theory of time-uniform inference. Motivated by bandit problems in which the number of arms may be large relative to the sample size, we extend the asymptotic confidence sequences of \citet{waudbysmith2024time} and \citet{bibaut2024near} to accommodate simultaneous monitoring across many adaptively sampled processes. In particular, we establish a time-uniform moderate deviations principle for the sequential $t$-statistic, which provides a sufficient approximation for valid multiple testing in settings where the number of hypotheses grows rapidly with the horizon. The refinement leads to a procedure that is robust under strategic sampling while retaining power in many-armed bandit designs.
\subsection{Outline of the paper}

The remainder of the paper is organized as follows. Section~2 introduces the setup and notation for our class of demonstration experiments, framing the problem in terms of multi-armed bandits with sub-Gaussian outcomes. Section~3 develops two test statistics---the pooled and max statistics---that remain valid under strategic sampling and, in some cases, early stopping. Section~4 presents the SN-UCB algorithm and analyzes its performance, showing how adaptive allocation can enhance power while controlling type~I error. Section~5 reports Monte Carlo simulations that illustrate the finite-sample properties of our methods and highlight their practical benefits. Section~6 concludes with a discussion of extensions and open directions for future research.

\section{Setup and Notation}
\label{sec:setup}

We consider a general experimental framework in which an experimenter sequentially allocates samples across a finite set of arms. Each arm may correspond to a treatment variant, a mutually exclusive subpopulation, or another experimental option of interest. At each round, the experimenter selects one arm and observes a noisy outcome. The objective is not to estimate all arm means precisely, nor to identify the single best arm, but rather to test whether any arm has a mean exceeding a given benchmark.

Formally, this corresponds to a standard multi-armed bandit setting with $k$ arms indexed by $g \in \{1, \ldots, k\}$ (see e.g.~\citealp{bubeck2012regret}), which we make explicit in \Cref{assn:dgm}.

\begin{assumption}[Data generating process] \label{assn:dgm}
	We consider a multi-armed bandit with $k$ arms: at each time  $t=1,\dots,T$ the experimenter selects an arm $g_t$ and observes the corresponding entry $X_{g_t}(t)$ of the potential outcomes vector $X(t)$. The vectors $X(t)$ are independently and identically distributed; its marginals $F_g$ have mean $\mu_g$ and variance $\sigma_g^2$.
\end{assumption}

We denote by $N_{g}(t)$ the number of times arm $g$ has been drawn up to and including round $t$.
We also let $T_g \subset [T]$ be the set of rounds $t$ for which $g_t = g$, and write $N_g = N_g(T) = \#T_g$ to denote the total number of times arm $g$ was chosen in the $T$ rounds.

The information available to the researcher at round $t$ consists of both the sequence of arms chosen up to time $t$ and the resulting data: $g_1,$ $X_{g_1}(1),$ $g_2,$ $X_{g_2}(2)$, and so on. 
Correspondingly, we define the $\sigma$-algebra generated at time $t$, summarizing all available information for the researcher up to that round:
\[
	\mathcal{F}_t = \sigma(\{g_i, X_{g_i}(i) : 1 \le i \le t\}).
\]
We require that $g_t$ is $\mathcal{F}_{t-1}$-measurable, for all $t \ge 1$.

In this paper, we are particularly interested in testing the null hypothesis that the mean of each arm $g$ falls below a corresponding, fixed threshold $u_g \in \mathbb{R}$. Comparison to a fixed threshold is natural when the status quo is measured precisely enough to be treated as known: in modern experimental settings, exposing units to unproven treatments, rather than measuring the incumbent, is often the binding cost. In this way, $u_g$ plays the role of a known control counterfactual. 

\begin{definition} \label{def:null}
The null and alternative hypotheses, denoted $\mathcal{H}_0(u)$ and $\mathcal{H}_1(u)$ for $u \in \mathbb{R}^k$, are as follows:
\[
	\mathcal{H}_0(u) = \left\{\max_{1 \le g \le k} \mu_g - u_g \le 0\right\},
	\quad 
    \text{and}
    \quad
	\mathcal{H}_1(u) =\left\{\max_{1 \le g \le k} \mu_g - u_g > 0\right\}.
\]
When $u = (0,\ldots,0)$, we simply write $\mathcal{H}_0$ or $\mathcal{H}_1$.
\end{definition}

Finally, we state two assumptions which will be maintained throughout the paper. \Cref{assn:sample-once}  concerns the sampling strategy, while \Cref{assn:sub-g} concerns the distribution of rewards. 
\begin{assumption}\label{assn:sample-once}
    The sampling strategy selects each arm twice at the outset, i.e.,~$N_g(2k) = 2$ for all arms $1 \le g \le k$.
\end{assumption}
This ensures that the arm-wise sample mean and variance is well defined.
\begin{assumption}\label{assn:sub-g}
    Each arm's distribution is sub-Gaussian with parameter $\nu \sigma_g$: for some $\nu \ge 1$, all $1 \le g \le k$, and all $s \in \mathbb{R}$, $\log\mathbb{E}\exp\{s(X_g(1)- \mu_g)\} \le s^2\nu^2\sigma^2_g/2$.
\end{assumption}
Note that Assumption \ref{assn:sub-g} is satisfied if (i) outcomes under each arm are bounded with variance bounded away from zero, or  (ii) if they are normally distributed. Without loss of generality, we take $\nu \ge 1$; $\nu$ appears in our bounds via a multiplicative constant, with polynomial dependence.

\section{Statistics that are robust to strategic sampling}
\label{sec:statistic}

In this section, we propose two test statistics which allow us to test the null hypothesis $\mathcal{H}_0(u)$ introduced in \Cref{def:null} at any pre-specified level $\alpha \in (0,1)$ for any strategic sampling algorithm satisfying \Cref{assn:sample-once}.
That is, when $\mathcal{H}_0(u)$ holds, the probability of rejection is at most $\alpha$. Note that since the parameter $u$ is up to the researcher, the problem can be reduced to that of testing $\mathcal{H}_0 = \mathcal{H}_0(0)$ by replacing each draw $X_g(t)$ by $X_g(t) - u_g$, preserving each arm's variance. Thus, without loss of generality, we focus the remainder of the discussion on $\mathcal{H}_0$.


\subsection{Pooled Testing} \label{sec:pooled}
The first statistic we consider works by pooling together information from all arms.
Before formally introducing it in \Cref{sec:pooled_defn}, we build intuition by considering the following infeasible statistic:
\[
	H_T =
		\frac{1}{\sqrt{T}} \sum_{t = 1}^T \frac{X_{g_t}(t)}{\sigma_{g_t}}
	 .
\]
Here, $H_T$ is infeasible because it depends on the unknown quantity $\sigma_{g_t}^2$.
It can be seen as a weighted average of standardized sums for each arm, where the weights correspond to the relative number of times each arm has been sampled.\footnote{See \cref{eq:weighted-average-of-t-stats} below.}

Under the two-sided null that $\mu_g=0$ for all $1 \le g \le k$, one can show that the limiting distribution of $H_T$ does not depend on the sampling strategy, so that tests based upon $H_T$ have asymptotically correct size.
The intuitive reason for constructing a test based upon $H_T$ is that under $\mathcal{H}_0$, the sample sums produced by any bandit algorithm cannot be positive in expectation: no sampling strategy can ``win against'' the null.
We formalize this intuition in \Cref{lemma:martingale}.
\begin{lemma} \label{lemma:martingale}
Under the null hypothesis $\mathcal{H}_0$, $\sqrt{t}H_t$ is a supermartingale adapted to $\mathcal{F}_t$. That is, for any $t \ge 1$ and $v \le t$ it holds
\[
	\mathbb{E}[\sqrt{t}H_t \mid \mathcal{F}_{v}] \le \sqrt{v}H_{v}.
\]
\end{lemma}

An important aspect of $H_T$ is its random normalization by $\sigma_{g_t}$ (which is random because the choice of $g_t$ depends on previously observed data). This ensures that the limit distribution is pivotal: it does not depend on the unknown distributions of the arms, or on the sampling strategy chosen by the experimenter. In contrast, the quadratic variation of the usual sum $S_T = \sum_{t=1}^T X_{g_t}(t)$ can strongly depend on the sampling strategy, so that its limit distribution may be a nontrivial mixture distribution \citep{hall2014martingale}.

Focusing on $H_T$ instead of $S_T$ simplifies the Gaussian approximation problem on a technical level (see e.g. \citealt{mourrat2013rate, fan2013cramer}). As a consequence, our analysis allows arbitrary bandit algorithms\textemdash even those whose behavior may depend on the sample budget, $T$. In contrast, many statistics proposed for best arm identification place strong regularity conditions on the sampling strategy to ensure that the quadratic variation is stable, and rule out standard UCB policies (a detailed discussion is provided by \citealp{hadad2021confidence}).\footnote{This is because asymptotic inference on the mean of any individual arm requires that it is sampled sufficiently many times, and best arm identification requires simultaneous inference on the means of all arms.
}

Another heuristic benefit of $H_T$ is that, if the experimenter's goal is to reject the null by having $H_T$ exceed a pre-specified threshold, it should be optimal to sample the arm $g$ with the maximum value of $z_g = \mu_g/\sigma_g$ as frequently as possible. This naturally corresponds to an optimization problem with bandit feedback, which we formalize and study in Section \ref{sec:ucb}.
Our results imply that the (asymptotic) distribution of $H_T$ under any sampling scheme decomposes as a standard normal plus a predictable drift given by
\begin{equation}\sqrt{T} \sum_{g=1}^k \omega_gz_g; \qquad \omega_g \coloneqq\frac{N_g(T)}{T}. \label{eq:weighted-average-of-t-stats}
\end{equation}
In other words it is proportional to a weighted sum, where the weights $\omega_g$ reflect the fraction of samples allocated to the arm $g$. Thus, optimizing power amounts to a $k$-arm bandit problem where the rewards are given by the signal-to-noise ratios $z_g$.

\subsubsection{Feasible statistics and regularized variance estimates}\label{sec:pooled_defn}

Since the arm variances $\sigma_g^2$ are unknown, any feasible analogue of $H_T$ must estimate them from data. When an arm has received few samples---a necessary feature of adaptive experiments---the naive plug-in can be unstable, and some form of regularization is needed. Let
\[
\hat\sigma_g^2
= \frac{1}{N_g} \sum_{t \in T_g} \{X_{g_t}(t) - \hat\mu_g\}^2,
\qquad
\hat\mu_g = \frac{1}{N_g} \sum_{t \in T_g} X_{g_t}(t),
\]
and consider the two-parameter family of regularized estimators
\[
\tilde\sigma_g(\lambda,\rho)
= \hat\sigma_g + \frac{\lambda}{\sqrt{N_g}}\,\mathbf{1}\{N_g \le \rho\}.
\]
The associated pooled statistic and its centered version are
\[
\tilde H_T(\lambda,\rho)
= \frac{1}{\sqrt{T}} \sum_{t=1}^T \frac{X_{g_t}(t)}{\tilde\sigma_{g_t}(\lambda,\rho)},
\qquad
\tilde H'_T(\lambda,\rho)
= \frac{1}{\sqrt{T}} \sum_{t=1}^T
\frac{X_{g_t}(t) - \mu_{g_t}}{\tilde\sigma_{g_t}(\lambda,\rho)} .
\]

Stabilizing the pooled statistic under adaptive allocation presents several competing challenges. Arms may receive only a small number of samples, making variance estimates highly sensitive to sampling error; variance estimates that are spuriously small can produce highly unstable estimates of the arm-wise mean. Conversely, inflating the variance estimates may severely reduce power when an arm truly has low variance and high signal-to-noise ratio. Because sampling itself is data dependent, balancing these effects is theoretically delicate. We therefore focus on two limiting regularization regimes that address this tradeoff in different ways.

Rather than carrying $(\lambda,\rho)$ throughout the notation, we define two derived
statistics corresponding to padding-based and threshold-based regularization:
\[
\hat\sigma^{\mathrm{pad}}_g := \tilde\sigma_g(\lambda_{k,T},\infty),
\qquad
\hat\sigma^{\mathrm{thr}}_g := \tilde\sigma_g(\infty,\rho_{k,T}),
\]
with pooled statistics $\hat H_T^{\mathrm{pad}}$, $\hat H_T^{\mathrm{thr}}$ and centered counterparts
$\hat H_T^{\mathrm{pad}\,'}$, $\hat H_T^{\mathrm{thr}\,'}$ defined analogously.
The two regularization schemes stabilize the pooled statistic in different ways. Padding regularization (controlled by $\lambda$) acts by enlarging the denominator, inflating variance estimates when sample sizes are small. Threshold regularization (controlled by $\rho$) instead resembles a trimming rule, restricting attention to arms with sufficiently many samples to permit reliable Studentization. The appendix develops a simultaneous analysis covering both regimes, but does not eliminate the tension between their respective guarantees.

\begin{theorem}[CLT for padding-regularized pooled statistic]\label{thm:pooled-pad}
Suppose Assumptions~\ref{assn:sample-once} and~\ref{assn:sub-g} hold
and take $\lambda_{k,T}=\sqrt{\log(kT)}$.
Under $H_0$, the statistic $\hat H_T^{\mathrm{pad}}$ is stochastically dominated by
its centered counterpart $\hat H_T^{\mathrm{pad}\,'}$, which satisfies
\begin{equation}\label{eq:feasible-clt}
\sup_{u\in\mathbb R}
\left|
\mathbb P\{\hat H_T^{\mathrm{pad}\,'}\le u\}-\Phi(u)
\right|
\le
C_\nu
\frac{\log^{3/2}(kT)}{\sqrt T}
\sum_{g=1}^{k}\bigl(\sigma_g\vee\sigma_g^{-1}\bigr).
\end{equation}
\end{theorem}
In particular, the above regularization depends only on $(k,T)$ and yields a completely tuning-free implementation. In contrast, the following gives a better theoretical approximation, but requires a known bound on $\nu$.
\begin{theorem}[CLT for threshold-regularized pooled statistic]\label{thm:pooled-thr}
Suppose Assumptions~\ref{assn:sample-once} and~\ref{assn:sub-g} hold
and take $\rho_{k,T}=C_\nu\log(kT)$.
Then under $\mathcal{H}_0$,
\[
\sup_{u\in\mathbb R}
\left|
\mathbb P\{\hat H_T^{\mathrm{thr}\,'}\le u\}-\Phi(u)
\right|
\le
C_\nu
\frac{k\log^{3/2}(kT)}{\sqrt T}.
\]
\end{theorem}

\noindent\emph{Proof of Theorems~\ref{thm:pooled-pad} and~\ref{thm:pooled-thr}.}
See Appendix~\ref{sec:proof-pooled}. \hfill$\square$

\medskip

These finite-sample approximation results imply asymptotic validity of pooled testing in sequences of experiments with a growing number of arms: the natural test $\{\hat H_T > c_\alpha\}$ controls size at level $\alpha + o(1)$, as formalized in the corollary below.

\begin{corollary}[Asymptotic validity of pooled testing]\label{cor:pooled-validity}
Suppose Assumptions~\ref{assn:sample-once}--\ref{assn:sub-g} hold and consider
a sequence of experiments for which
$
\bigl[\sum_{g\le k}(\sigma_g\vee\sigma_g^{-1})\bigr]^2\log^3(kT)/T\to0
$
as $T \to \infty$.
Let $c_\alpha$ denote the $(1-\alpha)$-quantile of the standard
normal distribution.
Then, under $H_0$,
\[
\limsup_{T\to\infty}
\mathbb P\{\hat H_T^{\mathrm{pad}}>c_\alpha\}
\le\alpha,
\qquad
\limsup_{T\to\infty}
\mathbb P\{\hat H_T^{\mathrm{thr}}>c_\alpha\}
\le\alpha.
\]
Moreover, under the two-sided null hypothesis that $\mu_g = 0$ for all $g \in \{1, \ldots, k\}$, this claim can be strengthened to
\[
\lim_{T\to\infty} \mathbb P\{\hat H_T^{\mathrm{pad}} > c_\alpha\}
= \lim_{T\to\infty} \mathbb P\{\hat H_T^{\mathrm{thr}} > c_\alpha\}
= \alpha.
\]
In this sense, the pooled test is non-conservative.
\end{corollary}

The main challenge of Theorems~\ref{thm:pooled-pad} and~\ref{thm:pooled-thr} is showing that the cumulative error of approximating $\sigma_g$ by $\hat\sigma_g$ for all $g \in \{1,\ldots,k\}$ is negligible. To do so we must account for the fact that the infrequently sampled arms $g$, for which $\sigma_g$ is poorly estimated, do not contribute much to the sum in the definition of $\hat H_T$. This crucially exploits the difference between our problem and that of best-arm identification, where \emph{all} means $\mu_g$ must be estimated simultaneously.

\begin{remark}[Comparison of regularization strategies] \normalfont
Theorems~\ref{thm:pooled-pad} and~\ref{thm:pooled-thr} reflect two complementary approaches to stabilizing the pooled statistic. Padding preserves contributions from all arms and yields a tuning-free procedure whose implementation depends only on observable problem dimensions. Thresholding instead excludes arms with insufficient samples, leading to a sharper Gaussian approximation and, as illustrated in Section~\ref{sec:sims}, improved empirical power in several settings. These advantages come with opposing tradeoffs. Padding may be conservative when signal-to-noise ratios are large, whereas thresholding introduces an additional tuning parameter. Developing a regularization strategy that combines the tuning-free implementation of padding with the sharper guarantees of thresholding remains an interesting direction for future work.
\end{remark}


\subsection{Max Statistic}

As shown in \Cref{sec:pooled}, the pooled statistic allows sharp tests of the hypothesis that no arm's mean exceeds a given threshold. 
However, the pooled testing approach has a number of limitations. 
For one, although it anticipates strategic sampling, it does not support valid inference if a researcher wishes to terminate the experiment early. It also does not allow the researcher to test stronger hypotheses on the means of individual arms: for example, if one arm $g^*$ clearly outperforms the others, the pooled statistic will not allow the researcher to reject the more specific null hypothesis $\mathcal{H}_0^{(g^*)}=\{\mu_{g^\star} \le 0\}$.
In this case the pooled statistic may also be significantly smaller than the simple $t$-statistic corresponding to $g^*$, due to its inclusion of samples from the other arms.

To overcome these limitations, we describe an alternative class of statistics that jointly test the individual hypotheses $\mathcal{H}_0^{(g)}=\{\mu_g \le 0\}$, for $g=1,\dots, k $. 
We do so by creating a test that looks at the  $t$-statistic of each individual arms, restricting to those that have been sampled sufficiently many times. 
This approach addresses the limitations introduced above.  
In our analysis, treating each arm separately also results in a weaker restriction $k \ll T$ on the number of arms, up to logarithmic factors. This is nearly as good as possible under Assumption \ref{assn:sample-once}, which requires $k \le T/2$.

One drawback of this approach in comparison to pooled testing is that the resulting tests are conservative: under the null hypothesis, they reject with asymptotic probability strictly less than $\alpha$. 
Still, they appear to be more powerful in scenarios where one arm performs much better than the rest.
Moreover, we show theoretically that the price is not much greater than a necessary multiple hypothesis correction for testing the means of $k$ arms. 

In order to introduce our testing procedure, we define the Student's $t$-statistic corresponding to the arm $g$, evaluated after $q$ samples have been drawn from arm $g$. Let $T_{g,q} \subseteq T_g$ denote the set of the first $q$ rounds at which arm $g$ is sampled, i.e., the $q$ smallest elements of $T_g$. The arm-wise $t$-statistic is then
\begin{equation}\label{eq:def-arm-t-stat}
    \hat Z_g(q) \coloneqq
        \frac{{\sum_{s\in T_{g,q} } X_{g_s}(s)}}
        {\sqrt{\sum_{s \in T_{g,q}} [X_{g_s}(s) - \hat\mu_{g,q}]^2}},
\end{equation}
where $\hat \mu_{g,q} = q^{-1}\sum_{s \in T_{g,q}} X_{g_s}(s)$ is the empirical mean of the first $q$ samples from arm $g$. 
Note, further, that at the end of the experiment, at least one arm has been sampled at least $m = \lceil T/k \rceil$ times. In this setting, we can account for strategic allocation of samples to arms using a version of the invariance principle for the sequence $\{\hat Z_g(q)\}_{q \ge 1}$, building upon \citet{waudbysmith2024time} and \citet{bibaut2024near}.

The test we propose compares the maximum of the arm-wise $t$-statistics $\hat Z_g(N_g(t))$, taken across sufficiently sampled arms, to a Robbins-type time-uniform boundary. The remainder of this subsection develops the machinery this requires: \Cref{lem:robbins-boundary} states the boundary-crossing probabilities for a single Brownian path, and the moderate-deviations extension that follows is what enables simultaneous control across many arms.

In particular, extending the results of \citet{robbins1970boundary} for sample sums, \citet{waudbysmith2024time} provided a general framework for approximating the time-uniform rejection probabilities $\mathbb{P}\{ \max_{q \ge q_0}\sqrt{q}|\hat Z_g(q)| > c(q)\}$, $q_0 \ge 0$, by corresponding (pivotal) boundary crossing probabilities for Brownian motion. In our paper we focus on the following {one-sided} boundary crossing probabilities.
\begin{lemma}[{\citealp[Examples 2 and 3]{robbins1970boundary}}]\label{lem:robbins-boundary}
For $q \ge 1$, let $S_q = \sum_{i=1}^q X_i$ denote the $q^\th$ partial sum from an i.i.d.~sequence $(X_1, X_2, \ldots)$ with $\mathbb{E}[X_1]=0$ and $\mathbb{E}[X_1^2]=1$. Let $\Phi(x)$ denote the standard Gaussian CDF. Then,
\begin{equation}\label{eq:asymptotic-anytime-distribution-lin}
\lim_{q_0 \uparrow \infty} \mathbb{P}\left\{\max_{q \ge q_0} S_q > \frac{aq}{\sqrt{q_0}} \right\} = 2[1 - \Phi(a)].
\end{equation}
Moreover, if we put $h(x) = x^2 + 2\log\Phi(x)$ and $\Psi^+(a) = 1-\Phi(a) + \vp(a)[a + \vp(a)/\Phi(a)]$, then we also have
\begin{equation}\label{eq:asymptotic-anytime-distribution-log}
\lim_{q_0 \uparrow \infty} \mathbb{P}\left\{\max_{q \ge q_0} S_q > \sqrt{q}h^{-1}[\log(q/q_0) + h(a)]\right\} = \Psi^+(a).
\end{equation}
\end{lemma}

To apply these approximations across a large number of arms, $k$, we extend the approximations of \citet{waudbysmith2023distribution, waudbysmith2024time} to a moderate deviations principle. The formal results are given in Propositions \ref{prop:mdp-lin} and \ref{prop:mdp-log}, within Appendix \ref{sec:max-proof}. Although for the purposes of this paper we focus on two particular time-uniform tests under rather strong moment assumptions, we believe that the extension may be of independent interest, as it formally justifies time-uniform testing on a large (relative to the sample size) number of sample means. The proof technique may be extended to other time-uniform boundaries and weaker moment assumptions. 



Lemma \ref{lem:robbins-boundary} motivates a test which considers the maximum of $\hat Z_g$ of all arms with sufficiently many samples, and compares it to a critical value $t^*_\alpha$ defined using the right-hand side of \cref{eq:asymptotic-anytime-distribution-lin,eq:asymptotic-anytime-distribution-log}. 
In particular, let $\mathfrak{K}(t,\zeta)$ denote the set of arms $g$ such that $N_g(t) \ge T_\zeta = \zeta T / k$ for some $\zeta \ge 1$. Note that $\mathfrak{K}(T,1)$ is always nonempty by a simple counting argument. 
\begin{theorem}\label{thm:max-test}
Let $\Phi(x)$ and $\Psi^+(x)$, and $h(x) = x^2 + 2\log\Phi(x)$ be defined as in Lemma \ref{lem:robbins-boundary}. For any $k \ge 1$ and $\alpha \in (0,1)$, there exist unique  numbers $ z_\alpha(k), w_\alpha(k)> 0$ such that
\[2k\left[1 - \Phi\left(z_\alpha\right)\right] = k\Psi^+(w_\alpha) = \alpha.\footnote{Under sampling with replacement (\Cref{assn:dgm}), these critical values may be slightly sharpened from the Bonferroni correction $k[1-F(x)]=\alpha$ to the ``independent'' correction $1 - F(x)^k = \alpha$.}
\]
Define the tests
\begin{align}
    A_{\textnormal{lin}} &= 
        \mathbbm{1}\left\{\max_{t \ge 1} \max_{g \in \mathfrak{K}(t,\zeta)} \hat Z_g(N_g(t)) -  z_\alpha\sqrt{\frac{N_g(t)}{\zeta T/k}} > 0\right\}; \label{eq:max-clt-lin}\\
    A_{\textnormal{log}} &= \mathbbm{1}
        \left\{\max_{t \ge 1} \max_{g \in \mathfrak{K}(t,\zeta)} \hat Z_g(N_g(t)) - h^{-1}\left(\log\left[\frac{N_g(t)}{\zeta T/k}\right] + h(w_\alpha)\right) > 0\right\}. \label{eq:max-clt-log}
\end{align} 
Then, in a sequence of experiments indexed by $T$ satisfying assumptions \ref{assn:sample-once} and \ref{assn:sub-g}, and such that $T/[k_T\log^{4}(Tk_T)] \uparrow \infty$, the following claims hold true.
\begin{enumerate}[label=(\roman*)]
    \item Under the null hypothesis that $\max_g\mu(g) \le 0$, the tests $A_{\textnormal{lin}}$ and $A_{\textnormal{log}}$ control type-I error at level $\alpha$: $\limsup_{T \to \infty}\bb{E}(A_{\textnormal{lin}}) \le \alpha$ and $\limsup_{T \to \infty}\bb{E}(A_{\textnormal{log}}) \le \alpha$.
    \item In general, for all $\zeta T/k \le q \le T$, we have $h^{-1}\left(\log\left[\frac{q}{\zeta T/k}\right] + h(w_\alpha)\right) \le c_\alpha + C\sqrt{\log k}$, where $c_\alpha$ is the $1-\alpha$ quantile of the standard normal distribution. 
\end{enumerate}
\end{theorem}

\begin{remark}\normalfont
In order to consider settings where $k_T$ diverges rapidly with the sample size, $T$, we cannot rely directly upon the results of \citet{waudbysmith2024time}, who showed a ``Kolmogorov-type'' (or small-deviations) approximation of the form
\begin{equation} \label{eq:small-deviations-example}
    \mathbb{P}\{\exists\, q \ge q_0 : \hat Z(q) \ge \psi(\alpha,q,q_0) \} \le   F_\psi(\alpha) + o(1),
\end{equation} 
where $\psi$ is a time-dependent boundary and $F_\psi$ is the corresponding boundary crossing probability for a Brownian path. In our setting, we must extend the above to a ``Cram\'er-type'' (or moderate deviations) approximation of the form
\begin{equation} \label{eq:moderate-deviations-example}
    \mathbb{P}\{\exists\, q \ge q_0 : \hat Z(q) \ge \psi(\alpha,q,q_0) \} \le   [1+o(1)]F_\psi(\alpha).
\end{equation}
This is accomplished by 
combining the quantitative invariance principle of \citet{sakhanenko1984rate} with  a non uniform anti-concentration bound for $F_\psi$, which is derived analytically. The proof is given in Appendix \ref{sec:max-proof}.
\end{remark}

\begin{remark}[Conservativeness]\normalfont\label{rmk:conservativeness}
Conservativeness of the tests in Theorem \ref{thm:max-test} occurs due to the need to account for complex dependence between the individual arms $t$-statistics caused by strategic allocation of samples. In order to bypass this dependence, we consider time-uniform, sequential tests on the individual arms' $t$-statistics, which will remain valid for any number of strategically allocated samples. This, of course, comes at a cost. 

  
Finally, the probability distributions used to construct the tests $A_{\textrm{lin}}$ and $A_{\textrm{log}}$ are based on an infinite horizon, as opposed to the finite horizon, $T$. In particular, these tests will remain valid even should the experimenter choose to continue the experiment until some arbitrary future time $T_2 > T$ while peeking at the data. This also incurs a cost; however, the cost is small, as the probability under the null of a rejection after time $T$ is relatively small for large $T$.
\end{remark}

\section{Strategic sampling algorithms and power}
\label{sec:ucb}

Having established validity of the tests corresponding to the pooled and maximum statistics in \cref{eq:feasible-clt,eq:max-clt-lin,eq:max-clt-log}, it is natural to consider strategic sampling algorithms that aim to maximize the value of these statistics. 
By doing so, we may attempt to optimize power against various alternative hypotheses. 
To carry this out, we make the following observations.
\begin{enumerate}
    \item Maximizing the rejection probability of the statistics in \cref{eq:feasible-clt,eq:max-clt-lin,eq:max-clt-log} can be expressed as an online optimization problem with bandit feedback. 
    \item If we take the value of an arm to be its signal-to-noise ratio, $z_g \coloneqq \mu_g/\sigma_g$, the corresponding pseudo-regret $R_T = \sum_{t \le T} (\max_{g} z_g) - z_{g_t}$ and number of mistakes $E_T = \sum_{t \le T} \mathbbm{1}\{(\max_{g} z_g) > z_{g_t}\}$ can be linked to the value of each of the test statistics considered in Section \ref{sec:statistic}. 
\end{enumerate}
In light of these observations, we consider bandit algorithms which aim to minimize the pseudo-regret, $R_T$, and the number of mistakes, $E_T$. 

\subsection{The SN-UCB Algorithm}

The proposed bandit algorithm uses deviation bounds for Studentized sums to bound the signal-to-noise ratio of each arm. For this reason, we call the algorithm ``self-normalized upper confidence bound,'' or SN-UCB for short. 

To introduce the SN-UCB procedure, we first introduce the exploration function
\[
	\tau(n,t;\beta) = 4.5\,\nu^2\sqrt{\beta \log(t)/n},
\] which determines how we bias arms' measurements to prioritize exploration. Here $\nu$ corresponds to the tail bound of \Cref{assn:sub-g}, and $\beta > 2$ is a tuning parameter.
Compared to the standard upper confidence bound algorithm for a $k$-armed bandit \citep{bubeck2012regret}, we replace the empirical mean of the arm $g$ by its studentized counterpart, $\hat Z_g\{N_g(t)\}$, and replace the upper confidence bound for the mean by
\begin{equation}
    \hat U_g(t;\beta) = \frac{\hat Z_g\{N_g(t-1)\}}{\sqrt{N_g(t-1)}} + \left[1 + \frac{|\hat Z_g\{N_g(t-1)\}|}{\sqrt{N_g(t-1)}} \right]\tau\{N_g(t-1),t;\beta\}, \label{eq:sn-ucb-definition}
\end{equation}
which bounds the signal-to-noise ratio. This leads to the following procedure.

\begin{center}
\begin{minipage}{.7\linewidth}
  \begin{algorithm}[H]
\caption{\label{algo:sn-ucb} SN-UCB.}

\medskip
{ \spacingset{1}
\begin{description}
    \item[Step 1:] For $t = 1, 2, \ldots, k$, choose $g_{2t-1},g_{2t}=t$. 
    \item[Step 2:] For $t = 2k+1, 2k+2, \ldots T$ 
    or until $A_{\text{log}}$ in \cref{eq:max-clt-log} is $1$,
    choose $g_{t} \in \{1,\dots,k\}$ that maximizes $\hat U_g(t;\beta)$.
\end{description}
}
\end{algorithm}
\end{minipage}
\end{center}

\bigskip 

We are able to prove a probabilistic bound on the number of times Algorithm \ref{algo:sn-ucb} selects a suboptimal arm. Note that the bound given below depends on the distributions of each of the arms: thus, it allows us to establish lower bounds on the rejection probability that hold for suitably constrained  classes of alternatives. We state results in expectation for simplicity; we refer the reader to Appendix \ref{sec:proof-regret-thm}, which contains the proof of Theorem \ref{thm:sn-ucb-regret-bound}, for a high probability statement. 

\begin{theorem}\label{thm:sn-ucb-regret-bound} 
    Put $z^* = \max_{g'}z_{g'}$. For each arm $g$ we write $\Delta_g = z^* - z_g$. The pseudo-regret $R_T = \sum_{t \le T} z^* - z_{g_t}$ and number of errors $E_T = \sum_{t \le T}\mathbbm{1}\{z_{g_t} < z^*\} $ attained by Algorithm \ref{algo:sn-ucb} satisfy the following inequalities:
\[\mathbb{E}[R_T] \le C_\nu\beta \log T \sum_{g: z_g < z^*} \left\{\Delta_g + \frac{(1+|z_g|)^2}{\Delta_g}\right\}; \quad \mathbb{E}[E_T] \le C_\nu\beta \log T \sum_{g: z_g < z^*} \left\{1 + \frac{(1+|z_g|)^2}{\Delta_g^2}\right\}.\]
\end{theorem}

Theorem \ref{thm:sn-ucb-regret-bound} is established by modifying the high-probability analysis of the empirical UCB algorithm of \citet{audibert2009exploration}. The central challenge in its proof is that both the location and width of the confidence band depend on the unknown signal-to-noise ratio $z_g$, which must be replaced by the estimated quantity $\hat Z_g\{N_g(t)\} / \sqrt{N_g(t)}$. 
\begin{remark}[Implications for power] \label{rmk:power}\normalfont
Combining Theorem \ref{thm:sn-ucb-regret-bound} with Theorems \ref{thm:pooled-pad} and~\ref{thm:pooled-thr}, we find that the pooled statistic is approximately normal  with mean $\sqrt{T}(z^* - \frac{1}{T}\mathbb{E}[R_T])$. In this sense, the test is competitive with an oracle decision rule when $\mathbb{E}[R_T] \ll \sqrt{T}$. Similarly, given an optimal arm $g^*$, its $t$-statistic will be approximately normal with mean $\sqrt{T-\mathbb{E}[E_T]}z^*$, suggesting that the condition $\mathbb{E}[E_T] \ll T$ is roughly sufficient to compete with the oracle rule.

In view of the central limit theorems of \Cref{thm:pooled-pad,thm:pooled-thr} and the invariance principle that justifies \Cref{thm:max-test}, these characterizations can be made exact for local parameter sequences in which $\Delta_g \asymp T^{-1/2}$. However, in this case, neither condition is implied by the regret bound of \Cref{thm:sn-ucb-regret-bound} due to its inverse dependence on $\Delta_g$. A complete characterization of the tests' power therefore requires (i) characterizing the moderate and large deviations of the considered test statistics, and (ii) characterizing regret in  regimes where $\Delta_g \to 0$ at a slower rate. These important problems are left to future work.
\end{remark}

\section{Simulations}
\label{sec:sims}

We now turn from theory to simulations, asking whether the size and power properties of Sections~3 and~4 hold up at realistic sample sizes, and whether the regret bound of SN-UCB is meaningful under realistic alternatives and sample-size constraints. 

\subsection{Simulation Design}

We vary the number of arms $k \in \{5,10,20,50\}$ and horizon $T \in \{200, 500, 1000, 2000\}$. Each arm $g$ generates i.i.d.\ outcomes with mean $\mu_g$ and variance $\sigma_g^2$. We consider three scenarios:
\begin{description}
  \item[Null:] $\mu_g = 0$ and $\sigma_g = 1$ for all $g$.
  \item[Single spike:] $\mu_1 = \delta$, $\mu_g = 0$ for $g > 1$, and $\sigma_g = 1$ for all $g$.
  \item[Multi-scale:] $\mu_g = \delta g$ and $\sigma_g^2 = 2 g^3$, so that larger arms have higher means but also higher variance. Crucially, the arm with the largest mean ($g = k$) does not have the largest signal-to-noise ratio.
\end{description}
For each configuration we perform 1,000 replications.

We compare six sampling strategies: (i) uniform allocation, which samples each arm with equal probability; (ii) SN-UCB (Algorithm~\ref{algo:sn-ucb}), which targets the arm with the highest estimated signal-to-noise ratio; (iii) standard UCB \citep{auer2002finite}, which targets the arm with the highest estimated mean; (iv) UCB-V \citep{audibert2009exploration}, a variance-aware variant of UCB whose exploration bonus scales with $\hat\sigma_g$; (v) Thompson sampling \citep{thompson1933likelihood}; and (vi) an oracle that always samples the arm with the true highest signal-to-noise ratio. For each strategy, we evaluate the pooled statistic and both max statistics (linear and log boundaries).

Two benchmarks warrant comment, since they enjoy an additional advantage relative to the adaptive procedures. The oracle samples only the optimal arm, so inference reduces to a one-sample $t$-test, and uniform allocation can similarly use a Bonferroni-corrected $t$-test without further adjustment.

\subsection{Type I Error}

Table~\ref{tab:type1} reports empirical rejection rates under the null hypothesis at nominal level $\alpha = 0.05$, when the sampling strategy is a standard UCB policy. 
The pooled statistic maintains close to nominal size across all configurations, including settings where $k$ is large relative to $T$ (e.g., $k = 50$, $T = 200$). This is notable because our theoretical results (Theorems~\ref{thm:pooled-pad} and~\ref{thm:pooled-thr}) require $k^2 \log^3(kT) / T$ to be small, which is not realistic when $k = 50$ and $T = 200$. The empirical robustness suggests that the pooled statistic may be valid under weaker conditions than those we establish.

The max statistics are close to nominal or conservative, consistent with Theorem~\ref{thm:max-test}. We impose a qualification floor $q_0 = \max\{25, T/k\}$ which binds at the small-$T$, large-$k$ corner, leaving zero rejections at $(k,T) = (50, 200)$ (no arm ever qualifies). The linear boundary shows mild size inflation in large-$k$, large-$T$ cells (.082 at $(50, 2000)$, .093 at $(20, 500)$), where the asymptotic regime $T/[k\log^4(kT)] \to \infty$ of Theorem~\ref{thm:max-test} is least realistic; the log boundary, which grows more slowly, remains close to nominal at the same cells.

\begin{table}[b!]
\centering
\begin{tabular}{cc|cccc|cccc|cccc}

& & \multicolumn{4}{c|}{Pooled} & \multicolumn{4}{c|}{Max (Linear)} & \multicolumn{4}{c}{Max (Log)} \\
\midrule \midrule
& $T$ & 200 & 500 & 1000 & 2000 & 200 & 500 & 1000 & 2000 & 200 & 500 & 1000 & 2000 \\
\midrule
\parbox[t]{2mm}{\multirow{4}{*}{\rotatebox[origin=c]{90}{$k$}}}
& 5  & .064 & .042 & .044 & .041 & .058 & .051 & .039 & .051 & .025 & .018 & .016 & .013 \\
& 10 & .056 & .053 & .055 & .048 & .040 & .055 & .052 & .051 & .028 & .034 & .016 & .019 \\
& 20 & .042 & .053 & .053 & .053 & .000 & .093 & .067 & .063 & .009 & .053 & .027 & .027 \\
& 50 & .025 & .033 & .040 & .043 & .000 & .001 & .055 & .082 & .000 & .005 & .060 & .045 \\

\end{tabular}
\caption{\small\label{tab:type1}Empirical type~I error rates under the null hypothesis ($\mu_g = 0$ for all $g$) at nominal level $\alpha = 0.05$. Results are shown for the pooled statistic and both max statistics (linear and log boundaries) across varying numbers of arms ($k$) and sample sizes ($T$). The max-statistic tests use the qualification rule $q_0 = \max(25, T/k)$, equivalently $\zeta = \max(1, 25k/T)$, which keeps qualifying arms in the regime where the Studentized statistic is close to its standard-normal limit. }
\end{table}

\subsection{Power: Multi-Scale Alternative}
\label{sec:multiscale-power}

Figure~\ref{fig:multiscale-power} displays power curves under the multi-scale alternative, where $\mu_g = \delta g$ and $\sigma_g^2 = 2 g^3$. This design is challenging because the arm with the largest mean ($g = k$) has a lower signal-to-noise ratio than smaller arms. An algorithm that targets the highest mean will concentrate samples on a suboptimal arm for the purpose of our test.

Across all three statistics, SN-UCB tends to outperform uniform allocation and the other adaptive methods. At small effect sizes (e.g., $\delta = 0.4$), SN-UCB achieves power near 0.75 with the pooled statistic, compared to roughly 0.26 for uniform allocation and 0.44 for standard UCB. The advantage of SN-UCB is that it directly targets the signal-to-noise ratio, which determines the drift of both the pooled and max statistics under the alternative.
Standard UCB, UCB-V, and Thompson sampling perform worse than SN-UCB in this setting because they instead target the arm with the highest mean. In the multi-scale design, this leads them to oversample the high-variance arms, reducing power.

The oracle provides an upper bound on achievable power. Recall that the oracle uses a simple $t$-test without correction for adaptive allocation, so the gap between SN-UCB and the oracle reflects both the cost of not knowing the optimal arm and the cost of robustness to strategic sampling. Still, SN-UCB becomes competitive as $\delta$ increases, consistent with the regret bound in Theorem~\ref{thm:sn-ucb-regret-bound}.


\begin{figure}[htbp]
\centering
\includegraphics[width=0.95\textwidth]{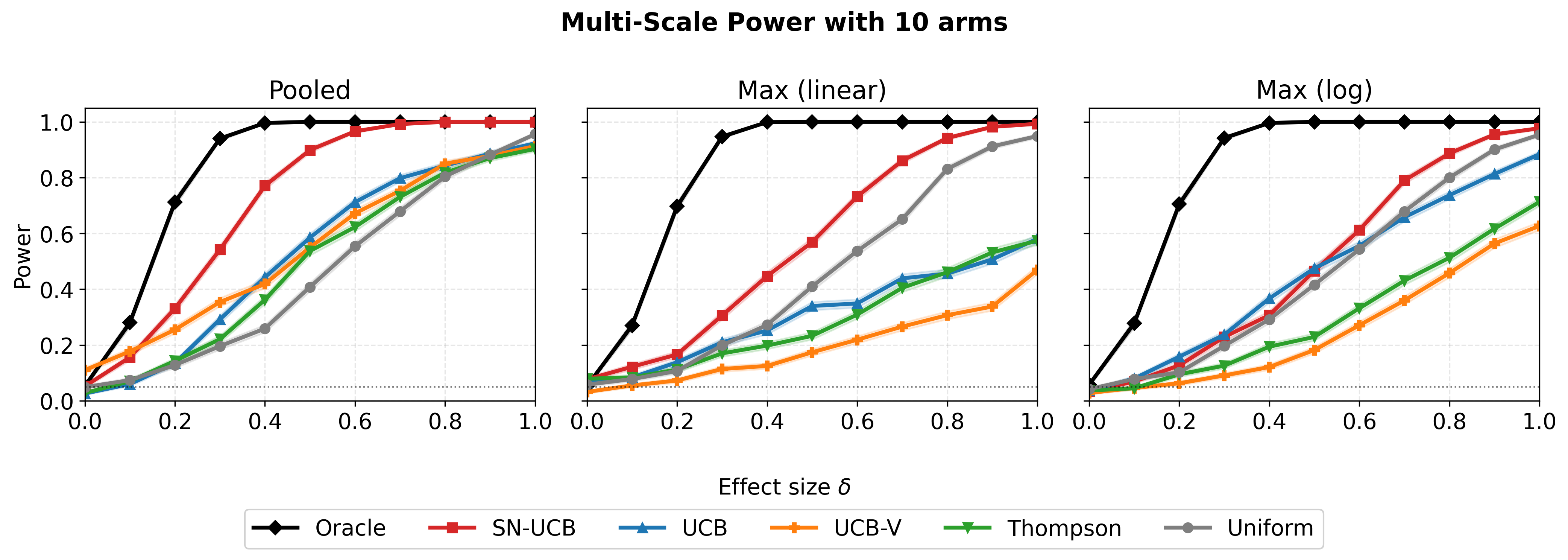}
\caption{\small Power curves under the multi-scale alternative ($\mu_g = \delta g$, $\sigma_g^2 = 2 g^3$) with $k=10$ arms, horizon $T=250$, and $\delta \in [0,1]$. Panels show the pooled, max-linear, and max-log statistics. Adaptive rules: SN-UCB (red), standard UCB (blue), UCB-V (orange), Thompson (green). Baselines: oracle ($t$-test on the highest-SNR arm; black) and uniform allocation with Bonferroni-corrected $t$-test (gray).}
\label{fig:multiscale-power}
\end{figure}

\subsection{Power: Single-Spike Alternative}
\label{sec:spike-power}

Figure~\ref{fig:spike-power} displays power curves under the single-spike alternative, where only one arm has a nonzero mean ($\mu_1 = \delta$) and all variances are equal ($\sigma_g = 1$). In this setting, the arm with the highest mean coincides with the arm with the highest signal-to-noise ratio.

The results differ qualitatively from the multi-scale case. For the pooled statistic, standard UCB and Thompson sampling now outperform SN-UCB, particularly at moderate effect sizes. This occurs because UCB and Thompson sampling more aggressively concentrate samples on the single best arm, whereas SN-UCB's exploration is more conservative as it accounts for uncertainty in the estimated variance.

For the max statistics, the differences among adaptive methods are smaller. All adaptive strategies substantially outperform uniform allocation, and the ranking among UCB, Thompson, and SN-UCB varies with $\delta$. Notably, uniform allocation with Bonferroni-corrected $t$-tests remains a reasonable baseline here, but adaptive methods still provide a meaningful improvement.

These results illustrate a key practical consideration: the relative performance of sampling algorithms depends on the signal structure. When the experimenter suspects that one arm dominates (single-spike), standard bandit algorithms may suffice, and max testing may be more appropriate. When effects are heterogeneous and the optimal arm is not obvious (multi-scale), SN-UCB's focus on signal-to-noise ratios provides a meaningful advantage, and pooled testing appears to be both more robust and more powerful.


\begin{figure}[b!]
\centering
\includegraphics[width=0.95\textwidth]{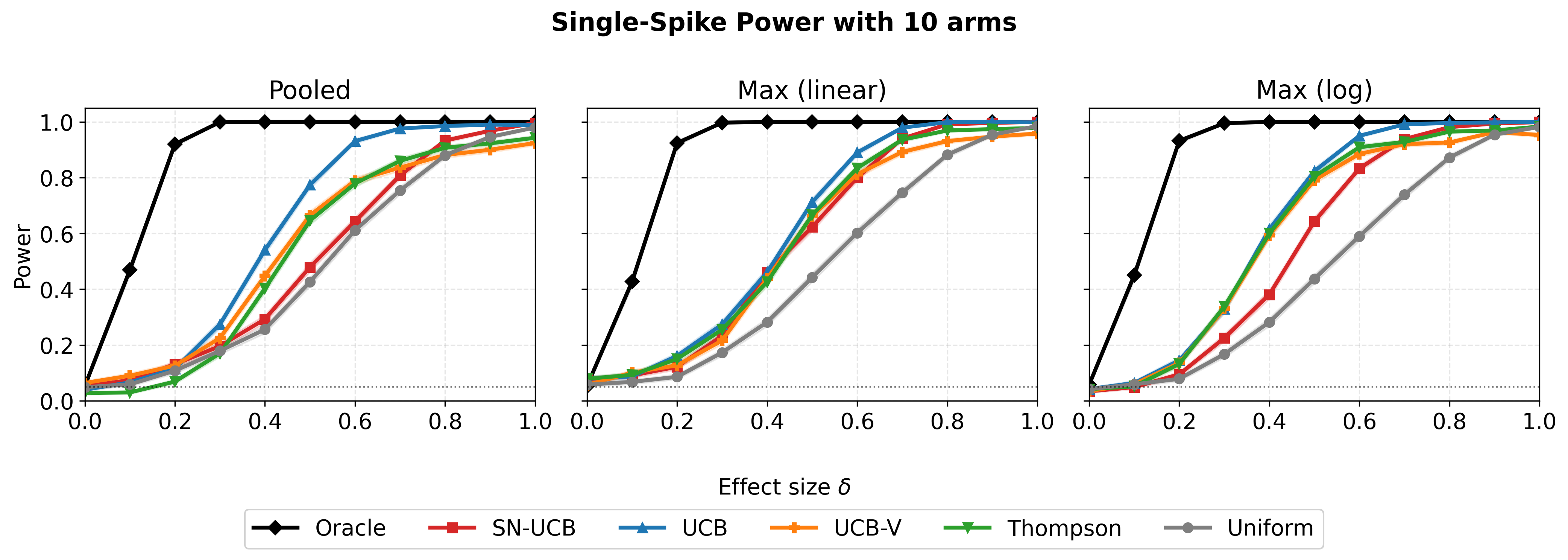}
\caption{\small Power curves under the single-spike alternative ($\mu_1 = \delta$, $\mu_g = 0$ for $g > 1$, $\sigma_g = 1$) with $k=10$ arms and horizon $T=250$. Panels show the pooled, max-linear, and max-log statistics. Adaptive rules: SN-UCB (red), standard UCB (blue), UCB-V (orange), Thompson (green). Baselines: oracle ($t$-test on the highest-SNR arm; black) and uniform allocation with Bonferroni-corrected $t$-test (gray).}
\label{fig:spike-power}
\end{figure}


\subsection{Application: Targeting an Unconditional Cash Transfer Program}
\label{sec:uct}

Unconditional cash transfers (UCTs)\textemdash direct payments to poor households, without restricting recipient behavior\textemdash have become a central tool of anti-poverty policy, and a growing body of evaluation evidence documents that they can meaningfully raise consumption and welfare in low-income settings \citep{haushofer2016short,bastagli2016cash}. When considering whether to scale such a program, policymakers typically must choose among variants proposed in the literature, which differ in how households are targeted and in the size of the transfer. We construct $k = 14$ arms to represent such a choice: each arm combines a \emph{targeting rule} with a \emph{transfer amount} (\$20 or \$50 per month). Eight core arms pair both transfer amounts with four widely-used targeting rules -- universal, rural, the bottom 20\% of the income distribution, and households with young children -- and six further arms cover three additional rules: community-based targeting, self-targeting via a costly application, and a categorical rule targeting the elderly. The outcome is the monthly consumption gain (USD), and the experimenter wishes to detect whether any variant clears a cost-effectiveness threshold of \$0.60 per dollar transferred. Means and variances are chosen to reflect magnitudes reported in the UCT evaluation literature \citep{haushofer2016short,merttens2013kenya,bastagli2016cash,blattman2014generating}; the choices are meant to be illustrative; Appendix~\ref{app:uct} reports the full parameter table and the studies that anchor each rule.

An important feature of the setting is that large average effects come bundled with large variance. Small transfers to tightly targeted, very poor households are spent on predictable necessities \textemdash food, medicine, basic household items\textemdash and so produce low-variance consumption changes. Broader targeting rules reach more heterogeneous households, and larger transfers give households room to make large, infrequent purchases (livestock, business equipment, home repair) in place of steady week-to-week consumption; both forces raise the mean response and its variance together. The arm with the largest net mean (a rural \$50 transfer, $\tilde\mu = \$10$, $\sigma = 45$, where $\tilde\mu_g$ denotes the consumption gain in excess of the arm's cost-effectiveness threshold) is therefore not the arm with the largest signal-to-noise ratio (an income-targeted \$20 transfer, $\tilde\mu = \$4$, $\sigma = 8$, $\text{SNR} = 0.50$).

Figure~\ref{fig:uct-power} reports power curves over $T \in \{50, 100, \ldots, 500\}$ for five sampling rules\textemdash SN-UCB, standard UCB \citep{auer2002finite}, variance-aware UCB \citep{audibert2009exploration}, Thompson sampling \citep{thompson1933likelihood}, and uniform allocation with a Bonferroni-corrected $t$-test\textemdash at 1{,}000 replications per cell. With the pooled statistic, SN-UCB overtakes uniform at $T \approx 150$ and reaches power $0.91$ at $T = 500$; the mean-based adaptive rules underperform because they concentrate on the largest-mean arm (a large transfer to rural houselholds) rather than the largest-SNR arm (a smaller income-targeted transfer). The picture shifts on the max-log statistic, where rapidly concentrating on a single high-mean arm is sufficient to demonstrate an effect, and standard UCB performs best. Interestingly, uniform allocation remains a reasonable benchmark in this study, especially at small sample-sizes. Allocation paths are reported in Appendix~\ref{app:uct}.


\begin{figure}[b!]
\centering
\includegraphics[width=0.95\textwidth]{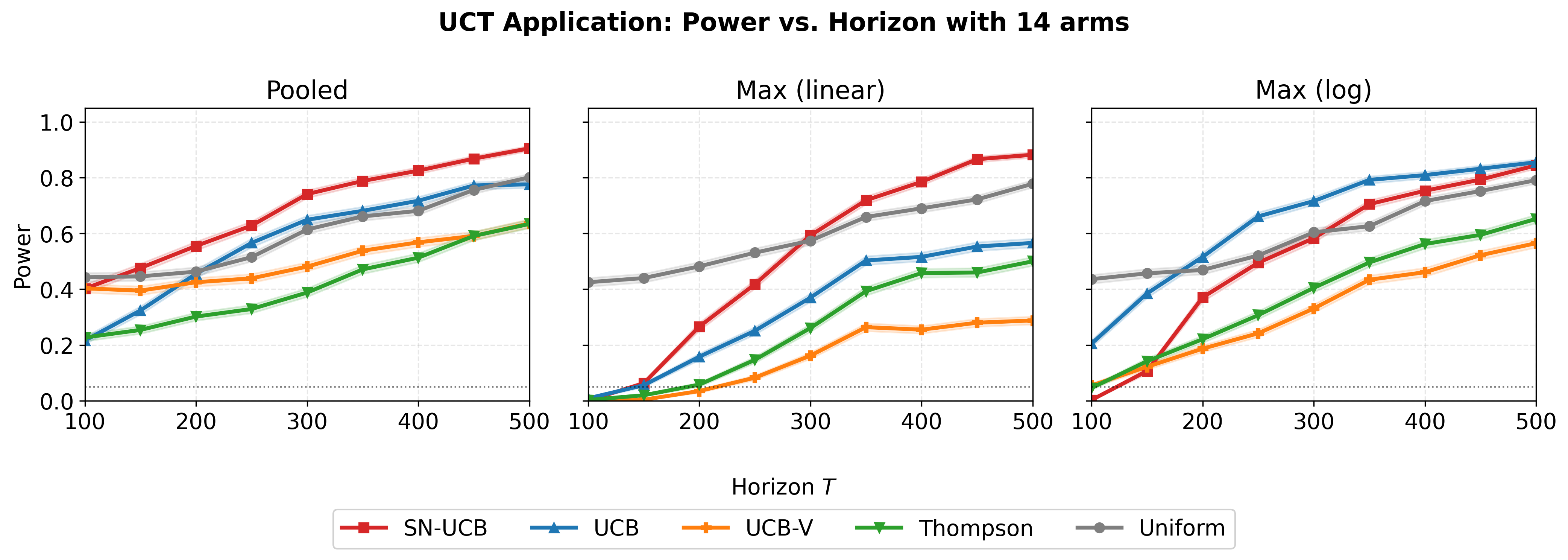}
\caption{\small Empirical rejection probability under the calibrated UCT alternative as a function of horizon $T$, with 1{,}000 Monte Carlo replications per cell. Panels show the pooled, max-linear, and max-log statistics; uniform allocation uses a Bonferroni-corrected one-sided $t$-test.}
\label{fig:uct-power}
\end{figure}


\subsection{Summary}

The simulations confirm the theoretical predictions of Sections~3 and~4. Both the pooled and max statistics control type~I error, with the pooled statistic achieving nominal size and the max statistics being conservative. Adaptive sampling improves power relative to uniform allocation, with the magnitude of improvement depending on the signal structure. Importantly, these gains are achieved despite the fact that adaptive methods must use test statistics that are robust to strategic allocation, whereas the oracle and uniform-allocation benchmarks can use simple $t$-tests. SN-UCB is particularly effective when the arm with the largest effect does not have the largest signal-to-noise ratio, as it directly optimizes the quantity that determines statistical power. When effects are concentrated in a single arm with equal variances, standard bandit algorithms perform comparably or better. The calibrated UCT application of Section~\ref{sec:uct} emphasizes that this multi-scale phenomenon is not a theoretical curiosity: it emerges naturally from the economics of household consumption response, where targeting precision and outcome variance are tightly linked.

\section{Conclusions}
\label{sec:conclusions}

This paper develops a framework for \emph{demonstration experiments}, where the goal is to establish whether any treatment arm exceeds a threshold rather than to estimate effects or identify the best arm. We propose two test statistics---pooled and max---that remain valid under adaptive sampling, and introduce the SN-UCB algorithm to optimize power by treating experimental design as stochastic optimization with bandit feedback.

Our results show that sharp inference is possible under nearly unrestricted adaptive sampling. The pooled statistic aggregates evidence across arms and achieves nominal size under the two-sided null. The max statistic is conservative but supports early stopping and focuses power on the most promising arm. Both tests require only that each arm is sampled twice initially, imposing no further restrictions on the sampling strategy. This contrasts with best-arm identification, which requires simultaneous inference on all arms and demands stronger design constraints.

From a practical standpoint, experimenters conducting exploratory studies can benefit substantially from adaptive designs. The pooled test suits settings where multiple arms may have moderate effects; the max test excels when one arm dominates. Using SN-UCB along with these tests offers power gains over uniform allocation without sacrificing type~I error control.

\subsection*{Future work}

Several questions remain open. A complete power characterization under semi-local alternatives requires analyzing moderate deviations of our statistics and the behavior of SN-UCB when gaps between arms shrink with sample size. Extensions to richer feedback structures---where multiple outcomes or overlapping subpopulations are observed---would broaden applicability. Although simulations suggest the pooled statistic performs well when $k$ is large relative to $T$, theoretical justification is lacking. \citet{kaufmann2018sequential} designed a ``Murphy sampling" algorithm which is shown to have good properties in the threshold bandit problem, and warrants further study in our context. \citet{sandoval2026multi} additionally study how to adaptively budget the type-I error allowance across arms within their e-process framework; extending such adaptive $\alpha$-allocation to our asymptotic, Brownian-motion-based setting would be an interesting and challenging direction for future work. 

\spacingset{1.25}
\bibliographystyle{plainnat}
\bibliography{bibliography}

\clearpage


\appendix

\numberwithin{theorem}{section}
\numberwithin{lemma}{section}
\numberwithin{proposition}{section}
\numberwithin{corollary}{section}
\numberwithin{definition}{section}
\numberwithin{assumption}{section}
\numberwithin{claim}{section}
\numberwithin{problem}{section}
\numberwithin{remark}{section}
\numberwithin{example}{section}
\numberwithin{solution}{section}

\spacingset{1.8}

\section{Section 3 Proofs}
\subsection{Probability Lemmas}

\subsubsection{Equivalent Representation of the DGP under \Cref{assn:dgm}}

We use the following fact extensively.

\begin{lemma}\label{lemma:iid-representation}
\lomtext{Under \Cref{assn:dgm}, the joint law of the observed data $\{(g_t, X_{g_t}(t))\}_{t=1}^T$ is identical to that of an experiment in which the $k$ arms are sampled from $k$ mutually independent i.i.d.\ sequences, with arm $g$ sampled from $F_g$.}
\end{lemma}

\begin{proof}[Proof of Lemma~\ref{lemma:iid-representation}]
Formally, let $\{Y_g(i)\}_{i \ge 1}$, $g = 1,\ldots,k$, be mutually independent i.i.d.\ sequences with $Y_g(i) \sim F_g$. Consider the alternative experiment that uses the same adaptive allocation rule but generates observations from these independent streams: when arm $g$ is selected for the $n$\textsuperscript{th} time, the observation is $Y_g(n)$. We show that the two experiments produce the same joint distribution of observed data.

We proceed by induction on rounds. For each $t \ge 1$, let
$\mathfrak{h}_t = (g_1, O_1, \ldots, g_t, O_t)$ denote the observed history
through round $t$, where $O_s$ is the observation at round $s$.
For the base case ($t=1$), note that the arm $g_1$ is $\mathcal{F}_0$-measurable. Under the original experiment the observation is
$X_{g_1}(1) \sim F_{g_1}$; under the alternative it is
$Y_{g_1}(1) \sim F_{g_1}$. The joint law of $\mathfrak{h}_1$ therefore agrees.

For the inductive step, suppose the joint law of $\mathfrak{h}_t$ is the
same under both experiments. Since $g_{t+1}$ is
$\mathcal{F}_t$-measurable, and $\mathcal{F}_t = \sigma(\mathfrak{h}_t)$,
$g_{t+1}$ has the same conditional distribution given $\mathfrak{h}_t$ under
both experiments. Given $\mathfrak{h}_t$ and $g_{t+1} = g$, let
$n = N_g(t) + 1$.
\begin{itemize}
    \item \emph{Original experiment.} The observation is $X_g(t+1)$,
    the $g$\textsuperscript{th} component of $X(t+1)$. Because the
    vectors $X(1), X(2), \ldots$ are i.i.d.\ (\Cref{assn:dgm}),
    $X(t+1)$ is independent of $X(1), \ldots, X(t)$, and hence
    $X_g(t+1)$ is independent of $\mathfrak{h}_t$, with marginal distribution~$F_g$.
    \item \emph{Alternative experiment.} The observation is $Y_g(n)$.
    By construction, $\{Y_g(i)\}_{i \ge 1}$ is i.i.d.\ $F_g$ and
    independent of all other randomness. Since $n$ is determined
    by $H_t$, the variable $Y_g(n)$ is independent of $\mathfrak{h}_t$,
    with distribution~$F_g$.
\end{itemize}
In both cases, $(g_{t+1}, O_{t+1})$ has the same conditional
distribution given $\mathfrak{h}_t$. Since $\mathfrak{h}_t$ itself has the same joint
law under both experiments, so does $\mathfrak{h}_{t+1}$.
\end{proof}

\subsubsection{Concentration}

We begin by stating some convenient moment bounds that follow from our \Cref{assn:sub-g}.
\begin{lemma}[Moment bounds]\label{lem:subg-to-bernstein}
Let $\xi_1, \xi_2, \ldots, \xi_n$ be a martingale difference sequence adapted to $\mathcal{F}_0, \mathcal{F}_1, \ldots, \mathcal{F}_n$, satisfying $\mathbb{E}\{\xi_i^2 \mid \mathcal{F}_{i-1}\} = 1$. If for some $\nu \ge 1$, all $i = 1,\ldots,n$, and all $\lambda \in \mathbb{R}$,
\[
    \log\mathbb{E}[e^{\lambda\xi_i}|\mathcal{F}_{i-1}] \leq \frac{\lambda^2\nu^2}{2}, \qquad \mathbb{E}[\xi_i^2|\mathcal{F}_{i-1}] = 1,
\]
then there exists a constant $C_\nu = O(\nu^3)$ depending only on $\nu$ such that for all $k \geq 3$,
\[
    \mathbb{E}[|\xi_i|^k|\mathcal{F}_{i-1}] \leq 2^{k/2}k\Gamma(k/2)\nu^{k} \le \frac{k!}{2}\,C_\nu^{k-2}.
\]
\end{lemma}
\begin{proof}
Combining the stated logarithmic moment-generating function bound with Markov's inequality gives sub-Gaussian concentration: $\mathbb{P}\{|\xi_i| > t \mid \mathcal{F}_{i-1}\} \le 2e^{-t^2/(2\nu^2)}$. We can integrate the tail bound to estimate, for any $k \ge 3$,
\[\mathbb{E}_{i-1}|\xi_i|^{k} = \int_0^\infty \mathbb{P}_{i-1}(|\xi_i|^{k} > t)\,dt =  \int_0^\infty \mathbb{P}_{i-1}(|\xi_i| > u)\,ku^{k-1}\,du \le 2k \int_0^\infty u^{k-1}e^{-u^2/(2\nu^2)}\,du,\] giving the moment bound $\mathbb{E}\{|\xi_i|^{k}\mid \mathcal{F}_{i-1}\} \le 2^{k/2}k\Gamma(k/2)\nu^{k}$. Finally, for $k \ge 3$ it holds that $k \Gamma(k/2) \le k!/2$ and $2^{k/2} \le (\sqrt 8)^{k-2}$, so we may take $C_\nu = \sqrt{8}\nu^3$.
\end{proof}


\begin{lemma}
Suppose $\Delta$ is a random variable such that $\mathbb{E}[\Delta^2] = \sigma^2$ and  $\log\mathbb{E}e^{\lambda\Delta} \le \nu^2\sigma^2\lambda^2/2$ for all $\lambda \in \mathbb{R}$. 
Define $K_4(\nu)$ as the largest $M \ge 1$ satisfying
\begin{equation}
    \cosh(\lambda M^{1/4}) - \frac{\lambda^2}{2}\bigl(M^{1/2} - 1\bigr) \;\le\; \exp(\nu^2\lambda^2/2) \qquad \forall \lambda > 0, \label{eq:jensen-implicit}
\end{equation}
and set $\kappa(\nu) \coloneqq K_4(\nu) - 1$. Then:
\[\mathbb{E}[\Delta^4] \le K_4(\nu)\sigma^4; \qquad \mathrm{Var}(\Delta^2) \le \kappa(\nu)\sigma^4 .\]
Finally, it holds that $K_4(\nu)/\nu^4 \le 6.8$ for all $\nu \ge 1$.
\end{lemma}
\begin{proof}
Write $Y \coloneqq \Delta/\sigma$, so that $\mathbb{E}Y^2 = 1$ and $\log\mathbb{E}e^{\lambda Y} \le \nu^2\lambda^2/2$. By symmetry of the MGF condition,
\[
    \mathbb{E}\cosh(\lambda Y) \;=\; \sum_{k=0}^\infty \frac{\lambda^{2k}}{(2k)!}\,\mathbb{E}Y^{2k} \;\le\; \exp(\nu^2\lambda^2/2)
\]
for all $\lambda \in \mathbb{R}$. For each $k \ge 2$, the function $x \mapsto x^{k/2}$ is convex on $[0,\infty)$, so Jensen's inequality applied to $X = Y^4$ gives $\mathbb{E}Y^{2k} = \mathbb{E}(Y^4)^{k/2} \ge (\mathbb{E}Y^4)^{k/2}$. Lower-bounding the Taylor series termwise, with $M \coloneqq \mathbb{E}Y^4$ and $\mathbb{E}Y^2 = 1$,
\[
    \mathbb{E}\cosh(\lambda Y_i) \;\ge\; 1 + \frac{\lambda^2}{2} + \sum_{k=2}^\infty \frac{\lambda^{2k}}{(2k)!}\,M^{k/2} \;=\; \cosh(\lambda M^{1/4}) - \frac{\lambda^2}{2}\bigl(M^{1/2} - 1\bigr),
\]
where the last equality uses $\sum_{k=2}^\infty u^{2k}/(2k)! = \cosh(u) - 1 - u^2/2$ at $u = \lambda M^{1/4}$. Combining with the MGF upper bound,
\begin{equation}
    \cosh(\lambda M^{1/4}) - \frac{\lambda^2}{2}\bigl(M^{1/2} - 1\bigr) \;\le\; \exp(\nu^2\lambda^2/2) \qquad \forall \lambda > 0, \label{eq:fourth-moment-bound}
\end{equation}
which is the implicit characterization of $K_4(\nu)$ in \eqref{eq:jensen-implicit}. Hence $\mathbb{E}\Delta^4 = \sigma^4\,\mathbb{E}Y^4 \le K_4(\nu)\,\sigma^4$, and $\mathrm{Var}(\Delta^2) = \mathbb{E}\Delta^4 - \sigma^4$. 

Finally, note that the left-hand side of \cref{eq:fourth-moment-bound} is increasing in $M$ while the right-hand side is constant. Thus, to upper-bound $K_4(\nu)$, it suffices to find some $\lambda$ and $M$ such that the inequality is violated. This is accomplished by taking $M = 6.8 \nu^4$ and $\lambda = 2/\nu$. 
\end{proof}
To establish consistency of the estimated mean $\hat{\mu}_g$ and variance $\hat{\sigma}^2_g$ for each arm $g$, we provide a self-contained argument using basic concentration tools.  
We refer the reader to \citet{pena2004self} and \citet{howard2020time} for many generalizations, refinements and extensions.

\begin{lemma}\label{lemma:iid-concentration}
Suppose $\{\Delta_1, \Delta_2, \dots\}$ is an i.i.d.~sequence with $\mathbb{E}\Delta_i = 0$, $\mathbb{E}[\Delta_i^2] = \sigma^2$ and  $\log\mathbb{E}e^{\lambda\Delta_i} \le \nu^2\sigma^2\lambda^2/2$ for all $\lambda \in \mathbb{R}$. Then:
 \begin{align}
      \mathbb{P}\left(\left|\sum_{i=1}^n \frac{\Delta_i}{\sigma} \right| \ge \nu \sqrt{2nx} \right) &\le 2e^{-x} \label{eq:concentration-standardized-mean};
      \\
      \mathbb{P}\left(\sum_{i=1}^n \frac{\Delta_i^2}{\sigma^2} - n \ge \nu^2\left\{ \sqrt{32nx} + 2x\right\}\right)
        &\le e^{-x}
        \label{eq:concentration-standardized-cov-ub};
        \\
      \mathbb{P}\left(n - \sum_{i=1}^n \frac{\Delta_i^2}{\sigma^2} \ge \frac{x}{3} + \sqrt{2\kappa(\nu)nx + \frac{x^2}{9}}  \right)
        &\le e^{-x}
        \label{eq:concentration-standardized-cov-lb}.
 \end{align}
 In particular, combining these yields for a universal constant $c > 0$
 \begin{equation}
       \mathbb{P}\left(\left|\sum_{i=1}^n \frac{\Delta_i^2}{\sigma^2} - n\right| \ge \delta n\right)
        \le 2\exp\left\{-c n\left(\frac{\delta^2}{2\nu^4} \wedge \frac{\delta}{\nu^2} \right)\right\}
        \label{eq:concentration-standardized-cov}.
 \end{equation}
\end{lemma}

\begin{proof}
\Cref{eq:concentration-standardized-mean} is the standard concentration inequality for sub-Gaussian variables \cite[Ch.~2]{boucheron2013concentration}. For \cref{eq:concentration-standardized-cov-ub}, put $Y_i = \Delta_i/\sigma$. Applying \Cref{lem:subg-to-bernstein} with $\xi_i = Y_i$ and $\mathcal{F}_{i-1}$ trivial gives the moment bound $\mathbb{E}|Y_i|^{2k} \le 2^{k+1}k!\nu^{2k}$. It follows that, for $0 \le \lambda \le 1/(2\nu^2)$, we have
\[\mathbb{E}[e^{\lambda(Y_i^2-1)}] \le e^{-\lambda}\left[1 + \lambda + \sum_{k \ge 2} 2^{k+1}\lambda^k\nu^{2k} \right] \le  e^{-\lambda}\left[1 + \lambda + \frac{8\nu^4\lambda^2}{1 - 2\nu^2\lambda}\right] \le \left[1  + \frac{8\nu^4\lambda^2}{1 - 2\nu^2\lambda}\right].\]
Since $\log(1+u) \le u$, this yields the bound \[ \log \mathbb{E}[e^{\lambda(Y_i^2-1)}] \le \frac{8\nu^4\lambda^2}{1 - 2\nu^2\lambda}.\]
Standard concentration for sub-gamma variables \cite[Ch.~2]{boucheron2013concentration} then gives \eqref{eq:concentration-standardized-cov-ub}.

For \eqref{eq:concentration-standardized-cov-lb}, set $W_i \coloneqq \sigma^2 - \Delta_i^2$, which satisfies $W_i \le \sigma^2$ a.s.\ (since $\Delta_i^2 \ge 0$), $\mathbb{E} W_i = 0$, and $\mathrm{Var}(W_i) = \mathbb{E}\Delta_i^4 - \sigma^4 \le \kappa(\nu)\sigma^4$ by \eqref{eq:fourth-moment-bound}. Bennett's inequality (in the Bernstein form, \citealp[Ch.~2]{boucheron2013concentration}) applied to $\sum_i W_i$ with $b = \sigma^2$ gives  \eqref{eq:concentration-standardized-cov-lb}. For \eqref{eq:concentration-standardized-cov}, we combine \cref{eq:concentration-standardized-cov-ub,eq:concentration-standardized-cov-lb} with a union bound and absorb constants, using $\kappa(\nu) \le 6.8\nu^4$. 
\end{proof}

    In \Cref{lemma:normalized-bound-mean} we give a crude time-uniform concentration inequality that is sufficient for our analysis. In effect, it accounts for dependence in strategic, sequentialy-sampled data by taking a simple union bound over prefixes. 

\begin{lemma}\label{lemma:normalized-bound-mean}
With probability at least $1-\eta$,
\[
    \left|\frac{1}{\sqrt{N_g}} \sum_{t \in T_g} \frac{X_g(t) - \mu_g}{\sigma_g}\right| \le \max_{1 \le q \le T} \left|\frac{1}{\sqrt{q}} \sum_{t=1}^q \frac{X_g(t) - \mu_g}{\sigma_g}\right|
    \lesssim \nu\sqrt{\log(T/\eta)}.
\]
\end{lemma}
\begin{proof}
By the equivalent representation of \Cref{lemma:iid-representation} and a union bound, we have
\begin{align*}
    \mathbb{P}\left\{\left|\frac{1}{\sqrt{N_g}}\sum_{t \in T_g} \frac{X_g(t) - \mu_g}{\sigma_g}\right| > u\right\}
    &\le \mathbb{P}\left\{\max_{1 \le q \le T} \left|\frac{1}{\sqrt{q}} \sum_{t=1}^q \frac{X_g(t) - \mu_g}{\sigma_g}\right|> u\right\} \\
    &\le \sum_{q=1}^T\mathbb{P}\left\{\left|\frac{1}{\sqrt{q}}\sum_{t=1}^q \frac{X'_g(t) - \mu_g}{\sigma_g}\right| > u \right\},
\intertext{where $X'_g(t)$ are i.i.d.~copies of $X_g(t)$ independent of $\mathcal{F}_\infty$. Using the two-sided inequality \eqref{eq:concentration-standardized-mean} from Lemma \ref{lemma:iid-concentration} with $\Delta_t = (X'_g(t) - \mu_g)/\sigma_g$, we get the bound}
    \mathbb{P}\left\{\left|\frac{1}{\sqrt{N_g}}\sum_{t \in T_g} \frac{X_g(t) - \mu_g}{\sigma_g}\right| > u \right\}  &\le 2Te^{-cu^2/\nu^2} \\
\implies   \mathbb{P}\left\{\left|\frac{1}{\sqrt{N_g}}\sum_{t \in T_g} \frac{X_g(t) - \mu_g}{\sigma_g}\right| > C\nu\sqrt{\log(2T/\eta)}\right\} &\le \eta
\end{align*}
for $C = c^{-1/2}$ sufficiently large.
\end{proof}


\begin{lemma}\label{lemma:normalized-bound-var} 
With probability at least $1-\eta$, 
\begin{equation}
    \begin{split}
    \left|\frac{1}{\sqrt{N_g}}\sum_{t \in T_g} \frac{[X_g(t)-\mu_g]^2 - \sigma^2_g}{\sigma^2_g}\right| 
    &\le \max_{1\le q \le t}\left|\frac{1}{\sqrt{q}}\sum_{t=1}^q \frac{[X_g(t)-\mu_g]^2 - \sigma^2_g}{\sigma^2_g}\right| \\
    &\lesssim \nu^2\left(\sqrt{\log(T/\eta)} + \log(1/\eta)\right) 
    \end{split}
    \label{eq:two-sided-variance-bound}
\end{equation}
and
\begin{equation}
    \frac{1}{\sqrt{N_g}}\sum_{t \in T_g} \frac{[X_g(t)-\mu_g]^2 - \sigma^2_g}{\sigma^2_g} \ge \min_{1 \le q \le T} \frac{1}{\sqrt{q}}\sum_{t =1}^q \frac{[X_g(t)-\mu_g]^2 - \sigma^2_g}{\sigma^2_g}  \gtrsim -\nu^2 \sqrt{\log(T/\eta)}\label{eq:lower-variance-bound}.
\end{equation}
\end{lemma} 
\begin{proof}
By the equivalent representation of \Cref{lemma:iid-representation} and a union bound, we have
\begin{align*}
    \mathbb{P}\left\{\left|\frac{1}{\sqrt{N_g}}\sum_{t \in T_g} \frac{[X_g(t)-\mu_g]^2 - \sigma^2_g}{\sigma^2_g}\right| > u\right\} 
    &\le \sum_{q=1}^T\mathbb{P}\left\{\left|\frac{1}{\sqrt{q}}\sum_{t=1}^q \frac{[X'_g(t)-\mu_g]^2 - \sigma^2_g}{\sigma^2_g}\right| > u\right\}, \\
\intertext{where $X'_g(t)$ are i.i.d.~copies of $X_g(t)$ independent of $\mathcal{F}_\infty$. Using inequality \eqref{eq:concentration-standardized-cov} from Lemma \ref{lemma:iid-concentration} with $\Delta_t = X'_g(t)-\mu_g$, we get the bound}
    &\le \sum_{q=1}^T 2\exp\left\{-c\left( \frac{u^2}{\nu^4} \wedge\frac{u\sqrt{q}}{\nu^2} \right) \right\}.
\intertext{Noting that the term $u\sqrt{q}/\nu^2$ only appears when $q \le u^2/\nu^4$, this may be rewritten as}
    &\le 2T\exp\left\{-\frac{cu^2}{\nu^4}\right\} + \frac{u^2}{\nu^4}\exp\left\{-\frac{cu}{\nu^2}\right\}.
\intertext{Finally, using the fact $A + B \lesssim A \vee B$ and $x^2e^{-x} \le C' e^{-x/C'}$ for sufficiently large $C'$, we obtain the bound}
    &\lesssim T\exp\left\{-\frac{c'u^2}{\nu^4}\right\} \vee \exp\left\{-\frac{c'u}{\nu^2}\right\}.
\end{align*}
This gives 
\[ \mathbb{P}\left\{\left|\frac{1}{\sqrt{N_g}}\sum_{t \in T_g} \frac{[X_g(t)-\mu_g]^2 - \sigma^2_g}{\sigma^2_g} \right|> C'\nu^2\left[\sqrt{\log(C'T/\eta)} \vee \log(C'/\eta)\right] \right\} \le \eta\]
for $C'$ sufficiently large, proving \eqref{eq:two-sided-variance-bound}. Similarly using \eqref{eq:concentration-standardized-cov-lb} instead of \eqref{eq:concentration-standardized-cov} proves \eqref{eq:lower-variance-bound}.
\end{proof}


\subsubsection{Gaussian Approximation}

\begin{lemma}[Martingale CLT; Theorem 3.1 of \citealt{fan2013cramer}]\label{lem:martingale-clt}
Let $\xi_1, \xi_2, \ldots, \xi_n$ be a martingale difference sequence adapted to $\mathcal{F}_1, \mathcal{F}_2, \ldots, \mathcal{F}_n$, such that
\begin{enumerate}[label=(\roman*)]
    \item $\mathbb{E}[\xi_i^k|\mathcal{F}_{i-1}] \le \frac{1}{2}k!\epsilon^{k-2}\mathbb{E}[\xi_i^2|\mathcal{F}_{i-1}]$ for some $\epsilon \in (0,\frac{1}{2}]$ (Bernstein's condition), and
    \item $\left| \sum_{i=1}^n \mathbb{E}[\xi_i^2|\mathcal{F}_{i-1}] - 1\right| \le \delta$ for some $\delta \in [0,\frac{1}{2})$.
Then, 
\[
    \sup_{t \in \mathbb{R}}\left|\mathbb{P}\left\{\sum_{i=1}^n \xi_i \le t\right\} - \Phi(t)\right| \lesssim \delta + \epsilon\,|\log\epsilon|.
\]
\end{enumerate}
\end{lemma}
\begin{corollary}\label{cor:subg-martingale-clt}
Let $X_1, X_2, \ldots X_n$ be a martingale difference sequence adapted to $\mathcal{F}_1, \mathcal{F}_2, \ldots, \mathcal{F}_n$, such that $\mathbb{E}[X_i^2|\mathcal{F}_{i-1}]= 1$, and such that $\log\mathbb{E}[e^{\lambda X_i}|\mathcal{F}_{i-1}] \leq {\lambda^2\nu^2}/{2}$ for all $\lambda \in \mathbb{R}$.
Then, there exists a universal constant $C_\nu > 0$ depending only upon $\nu$ such that
\[
    \sup_{t \in \mathbb{R}}\left|\mathbb{P}\left\{\frac{1}{\sqrt{n}}\sum_{i=1}^n X_i \le t\right\} - \Phi(t)\right| \le \frac{C_\nu (1\vee\log n)}{\sqrt{n}}.
\]
\end{corollary}
\begin{proof}[Proof of Corollary \ref{cor:subg-martingale-clt}]
We will apply \Cref{lem:martingale-clt} with $\xi_i = X_i/\sqrt{n}$, $\delta = 0$, and $\epsilon = C'_\nu /\sqrt{n}$. 
By construction, $\mathbb{E}[\xi_i^2|\mathcal{F}_{i-1}] = n^{-1}$, so condition (ii) of \Cref{lem:martingale-clt} is automatically met. 
It therefore suffices to check Bernstein's condition. 
By \Cref{lem:subg-to-bernstein} applied to $X_i$ (which satisfies $\mathbb{E}[X_i^2|\mathcal{F}_{i-1}]=1$ by assumption), we have $\mathbb{E}[|X_i|^k|\mathcal{F}_{i-1}] \leq \frac{k!}{2}C_\nu^{k-2}$ for all $k \geq 3$. Dividing both sides by $n^{k/2}$ and using $\xi_i = X_i/\sqrt{n}$ and $\mathbb{E}[\xi_i^2|\mathcal{F}_{i-1}] = 1/n$, we obtain
\begin{equation*}
    \mathbb{E}[|\xi_i|^k|\mathcal{F}_{i-1}] = \frac{\mathbb{E}[|X_i|^k|\mathcal{F}_{i-1}]}{n^{k/2}}
    \leq \frac{k!}{2}\frac{C_\nu^{k-2}}{n^{k/2}}
    = \frac{k!}{2}\left(\frac{C_\nu}{\sqrt{n}}\right)^{k-2} \cdot \frac{1}{n}
    = \frac{k!}{2}\left(\frac{C_\nu}{\sqrt{n}}\right)^{k-2}\mathbb{E}[\xi_i^2|\mathcal{F}_{i-1}].
\end{equation*}
It follows that whenever $\epsilon \le \frac{1}{2}$, we may take $\epsilon = C'_\nu/\sqrt n$ for $C'_\nu = C_\nu$ and apply \Cref{lem:martingale-clt} to deduce the claimed bound.

In turn,  $\epsilon \le \frac{1}{2}$ clearly holds for sufficiently large $n \ge n_\nu$, where $n_\nu$ depends only upon $\nu$. 
This proves the result for all $n \ge n_\nu$ with $C_\nu = C'_\nu$.
On the other hand, for $n \le n_\nu$, we can choose $C''_\nu = \max_{1 \leq k \leq n_\nu} \frac{\sqrt{k}}{1\vee\log k}$ to ensure $\frac{C''_\nu (1\vee\log n)}{\sqrt{n}} \geq 1$, making the claimed bound trivially true since the Kolmogorov distance is at most $1$.
\end{proof}

We also state two lemmas that help us derive bounds on the Kolmogorov distance between our statistic and the standard normal distribution. 

\begin{lemma}[Lemma 2.1 of \citealp{chernozhukov2016empirical}]\label{lemma:coupling-to-kolmogorov}     
Suppose that $\mathbb{P}(|X-Y| > \nu) \le \eta$. Then 
\[
    \sup_{t \in \mathbb{R}}\left|\mathbb{P}(X \le t) - \mathbb{P}(Y \le t) \right| \le \eta + \sup_{t \in \mathbb{R}} \mathbb{P}(|Y - t| \le \nu). 
\]
If $Y$ has a standard Gaussian distribution then the bound simplifies to $\eta + 2\nu$.
\end{lemma}
\begin{proof} 
The first statement is exactly \citet[Lemma 2.1]{chernozhukov2016empirical};
the second claim follows as the Gaussian density is bounded by $1$.
\end{proof}

\begin{lemma}\label{lemma:combine-coupling-with-kolmogorov} 
Suppose that $\bb{P}(|X-Y| > \nu ) \le \eta$ and $\sup_{t \in \bb{R}}|\bb{P}(Y \le t) - \Phi(t)| \le \epsilon$. 
Then
\[
    \sup_{t \in \bb{R}}|\bb{P}(X \le t) - \Phi(t)| \le \eta + 3\epsilon + 2\nu.
\]
\end{lemma}
\begin{proof}
Note that given any $t$ our assumptions imply $\bb{P}(Y \le t) \le \Phi(t) + \epsilon$ and
\[
    \bb{P}(Y < t) = \lim_{u \uparrow t} \bb{P}(Y \le u) \ge \lim_{u \uparrow t} \Phi(u) - \epsilon = \Phi(t) - \epsilon,
\]
using the left-limit property of the distribution function $t \mapsto \bb{P}(Y \le t)$ and continuity of $\Phi$.  
Thus
\begin{align*}
    \sup_{t \in \mathbb{R}} \mathbb{P}(|Y - t| \le \nu) 
    &= \sup_{t \in \mathbb{R}} \mathbb{P}(Y  \le t + \nu) -  \mathbb{P}(Y  < t - \nu) \\
    &\le \sup_{t \in \mathbb{R}} \left\{\Phi(t + \nu) + \epsilon - [\Phi(t - \nu) - \epsilon]\right\} \le 2\epsilon + 2\nu,
\end{align*}
where we have used $|\Phi'| \le 1$ in the last step. 
We then use \Cref{lemma:coupling-to-kolmogorov} and the triangle inequality for supremum norms to conclude:
\begin{align*}
    \sup_{t \in \bb{R}}|\bb{P}(X \le t) - \Phi(t)|
    &\le \sup_{t \in \bb{R}}|\bb{P}(X \le t) - \bb{P}(Y \le t)| + \sup_{t \in \bb{R}}|\bb{P}(Y \le t) - \Phi(t)| \\
    &\le \eta + \sup_{t \in \mathbb{R}} \mathbb{P}(|Y - t| \le \nu) + \epsilon \le 3\epsilon + 2\nu + \eta. \qedhere
\end{align*} 
\end{proof}

\subsection{Proof of Theorems \ref{thm:pooled-pad} and~\ref{thm:pooled-thr}} \label{sec:proof-pooled}

Before proving Theorems~\ref{thm:pooled-pad} and~\ref{thm:pooled-thr}, we establish the supermartingale property of $H_t$ asserted in \Cref{lemma:martingale} and the asymptotic validity corollary \ref{cor:pooled-validity}.

\begin{proof}[Proof of \Cref{lemma:martingale}]
We show that under $\mathcal{H}_0$, the process $H_t = T^{-1/2}\sum_{s \le t} X_{g_s}(s)/\sigma_{g_s}$ is an $\mathcal{F}_t$-supermartingale. Since $g_s$ is $\mathcal{F}_{s-1}$-measurable (\Cref{assn:dgm}), so is $\sigma_{g_s}$. By the equivalent i.i.d.\ representation (\Cref{lemma:iid-representation}), the observation $X_{g_s}(s)$ has conditional mean $\mu_{g_s}$ given $\mathcal{F}_{s-1}$, so
\[
    \mathbb{E}\left[\left.T^{-1/2}\frac{X_{g_s}(s)}{\sigma_{g_s}}\right|\mathcal{F}_{s-1}\right]
    = T^{-1/2}\frac{\mu_{g_s}}{\sigma_{g_s}} \le 0 \quad \text{a.s.},
\]
because $\mu_g \le 0$ for every $g$ under $\mathcal{H}_0$ and $\sigma_g > 0$. Hence $\mathbb{E}[H_s \mid \mathcal{F}_{s-1}] \le H_{s-1}$. Iterating the tower property for $v < s \le t$ yields $\mathbb{E}[H_t \mid \mathcal{F}_v] \le H_v$ a.s.
\end{proof}

\begin{proof}[Proof of \Cref{cor:pooled-validity}]
For $\hat H_T^{\mathrm{pad}}$, note that $\hat H_T^{\mathrm{pad}} \le \hat H_T^{\mathrm{pad}\,'}$ a.s.\ under $\mathcal{H}_0$, since $\hat H_T^{\mathrm{pad}} - \hat H_T^{\mathrm{pad}\,'} = T^{-1/2}\sum_t \mu_{g_t}/\hat\sigma^{\mathrm{pad}}_{g_t} \le 0$ when $\mu_{g_t} \le 0$ and $\hat\sigma^{\mathrm{pad}}_{g_t} > 0$. Hence $\mathbb{P}\{\hat H_T^{\mathrm{pad}} > c_\alpha\} \le \mathbb{P}\{\hat H_T^{\mathrm{pad}\,'} > c_\alpha\}$, and by \Cref{thm:pooled-pad} the RHS is bounded by
\[
    1 - \Phi(c_\alpha) + C_\nu\frac{\log^{3/2}(kT)}{\sqrt T}\sum_{g=1}^k(\sigma_g \vee \sigma_g^{-1}) = \alpha + o(1),
\]
under the stated growth condition. The analogous argument for $\hat H_T^{\mathrm{thr}}$ uses \Cref{thm:pooled-thr}. Taking the $\limsup$ yields the first claim.

For the second claim, observe that under the sharp null $\mu_g \equiv 0$, the feasible statistic equals its centered counterpart identically: $\hat H_T^{\bullet} = \hat H_T^{\bullet\,'}$ a.s. Applying \Cref{thm:pooled-pad,thm:pooled-thr} directly to $\hat H_T^{\bullet}$ yields
\[
    \left|\mathbb{P}\{\hat H_T^{\bullet} > c_\alpha\} - \alpha\right| \to 0
\]
under the stated growth condition, giving the exact limit $\alpha$.
\end{proof}

\begin{lemma}\label{lem:consistent-estimator-variance}
With probability at least $1-\eta$, the estimator $\hat\sigma$ satisfies
\begin{align} 
\label{eq:estimator-consistency}
\left|\frac{\hat\sigma_g - \sigma_g}{\sigma_g}\right| &\lesssim \nu^2 \left( \frac{\log(T/\eta)}{N_g} + \frac{\sqrt{\log(T/\eta)} + \log(1/\eta)}{\sqrt{N_g}} \right)
\intertext{and}
\label{eq:estimator-lb}
\frac{\sigma_g - \hat \sigma_g}{\sigma_g} &\lesssim \nu^2 \left( \frac{\log(T/\eta)}{N_g} + \sqrt{\frac{\log(T/\eta)}{N_g} }\right).
\end{align}
\end{lemma}
\begin{proof}
In order to verify Lemma \ref{lem:consistent-estimator-variance}, we proceed in a few steps. To simplify notation, we now abbreviate $X_t \coloneqq X_{g_t}(t)$, put $\hat\mu_g = \frac{1}{N_g}\sum_{t \in T_g} X_t$, and note the identities
\begin{align*}
    \hat\sigma_g^2 - \sigma_g^2 
    &= \frac{1}{N_g}\sum_{t \in T_g} (X_t - \hat\mu_g)^2 - \sigma_g^2 \\
    &= \left\{\frac{1}{N_g}\sum_{t \in T_g} [X_t - \mu_g]^2 - \sigma_g^2\right\} - [\hat \mu_g - \mu_g]^2  \\
    &= \left\{\frac{1}{N_g}\sum_{t \in T_g} [X_t - \mu_g]^2 - \sigma_g^2\right\} - \left\{\frac{1}{N_g} \sum_{t \in T_g} X_t - \mu_g \right\}^2.
\intertext{Therefore,}
\frac{\hat\sigma_g^2 - \sigma_g^2}{\sigma_g^2} &=  \left\{\frac{1}{N_g}\sum_{t \in T_g} \frac{[X_t - \mu_g]^2 - \sigma_g^2}{\sigma_g^2}\right\} - \left\{\frac{1}{N_g} \sum_{t \in T_g} \frac{X_t - \mu_g}{\sigma_g}\right\}^2.
\end{align*}
Thus, it suffices to give high-probability bounds on the two bracketed terms. For the first term, Lemma \ref{lemma:normalized-bound-mean} gives that with probability $1-\eta$,
\[ \left|\frac{1}{N_g}\sum_{t \in T_g} \frac{X_t-\mu_g}{\sigma_g}\right| \lesssim \nu\sqrt{\frac{2\log(T/\eta)}{N_g}}.\]
Next, we may apply Lemma \ref{lemma:normalized-bound-var} to find that with probability $1-\eta$,
\[\left|\frac{1}{N_g}\sum_{t \in T_g} \frac{[X_g(t)-\mu_g]^2 - \sigma^2_g}{\sigma^2_g}\right| \lesssim \nu^2\left(\frac{\sqrt{\log(T/\eta)} + \log(1/\eta)}{\sqrt{N_g}}\right)\]
Combining these two estimates with a union bound, we obtain that with probability $1-\eta$,
\[\left|\frac{\hat\sigma_g^2 - \sigma_g^2}{\sigma^2_g}\right| \lesssim  \nu^2 \left( \frac{\log(T/\eta)}{N_g} + \frac{\sqrt{\log(T/\eta)} + \log(1/\eta)}{\sqrt{N_g}} \right).\]
Finally, we use the fact that $|1-u| \le |1-u||1+u| = |1 - u^2|$ for $u = \hat\sigma/\sigma > 0$ to conclude. Replacing our application of \eqref{eq:two-sided-variance-bound} with  \eqref{eq:lower-variance-bound} proves \eqref{eq:estimator-lb}.
\end{proof}

\begin{proposition}\label{prop:feasible-coupling}
Put \(\Delta_1 = \sqrt{\log(kT/\eta)} + \log(k/\eta)\) and \(\Delta_2 = \sqrt{\log(kT/\eta)}.\) With probability at least $1-\eta$, the centered, feasible statistic $\hat H'_T(\lambda,\infty)$ satisfies
\begin{align}|\hat H'_T(\lambda,\infty) - H'_T| \lesssim \frac{1}{\sqrt{T}}\sum_{g=1}^k \left( \nu^3\Delta_1\Delta_2 + \nu\Delta_2\frac{\lambda}{\sigma_g} + \nu^5\Delta_1\Delta^2_2 \frac{\sigma_g}{\lambda}\right). \label{eq:simple-reg-coupling}
\end{align} 
Moreover, for $\rho <\infty$, the statistic $\hat H'_T(\lambda,\rho)$ satisfies 
\begin{align} 
|\hat H'_T(\lambda,\rho) - H'_T| \lesssim \frac{C_\nu}{\sqrt{T}}\sum_{g=1}^k \rho^{3/2} + \frac{\sigma_g}{\lambda}\rho^2.
\label{eq:complex-reg-coupling}
\end{align}
with probability at least $1- kTe^{-\rho/(C\nu^4)}$.
\end{proposition}

\begin{proposition}\label{prop:feasible-coupling}
Put \(\Delta_1 = \sqrt{\log(kT/\eta)} + \log(k/\eta)\) and \(\Delta_2 = \sqrt{\log(kT/\eta)}\). With probability at least \(1-\eta\), the centered, feasible statistic \(\hat H'_T(\lambda,\infty)\) satisfies
\begin{align}
|\hat H'_T(\lambda,\infty) - H'_T|
\lesssim \frac{1}{\sqrt{T}}\sum_{g=1}^k
\left(
\nu^3\Delta_1\Delta_2
+\nu\Delta_2\frac{\lambda}{\sigma_g}
+\nu^5\Delta_1\Delta_2^2\frac{\sigma_g}{\lambda}
\right).
\label{eq:simple-reg-coupling}
\end{align}
Moreover, for \(\rho<\infty\), the statistic \(\hat H'_T(\lambda,\rho)\) satisfies
\begin{align}
|\hat H'_T(\lambda,\rho)-H'_T|
\lesssim
\frac{C_\nu}{\sqrt{T}}\sum_{g=1}^k
\left(
\rho^{3/2}+\frac{\sigma_g}{\lambda}\rho^2
\right)
\label{eq:complex-reg-coupling}
\end{align}
with probability at least \(1-kT\exp{-\rho/(C\nu^4)}\).
\end{proposition}

\begin{proof}
We first define a high-probability event controlling the mean and variance estimators for every arm, and then use these controls to couple the feasible and infeasible statistics.

\paragraph{Step 1: high-probability events.}
By a union bound over \(g\le k\), Lemmas \ref{lemma:normalized-bound-mean} and \ref{lem:consistent-estimator-variance} imply that, on an event \(\mathcal E(\eta)\) with probability at least \(1-\eta\), simultaneously for all \(g\le k\),
\begin{align}
\left|\frac{1}{\sqrt{N_g}}\sum_{t\in T_g}
\frac{X_g(t)-\mu_g}{\sigma_g}\right|
&\lesssim \nu\Delta_2, \nonumber\
\left|\frac{\hat\sigma_g-\sigma_g}{\sigma_g}\right|
&\lesssim \nu^2
\left(
\frac{\Delta_2^2}{N_g}+\frac{\Delta_1}{\sqrt{N_g}}
\right), \label{eq:feasible-event}\
\frac{\sigma_g-\hat\sigma_g}{\sigma_g}
&\lesssim \nu^2
\left(
\frac{\Delta_2^2}{N_g}+\frac{\Delta_2}{\sqrt{N_g}}
\right). \nonumber
\end{align}
Consequently, for a sufficiently large universal constant \(C_2\),
\[
\sqrt{N_g}\ge C_2\nu^2\Delta_2
\quad\Longrightarrow\quad
\hat\sigma_g\ge \sigma_g/2
\quad\text{and}\quad
\sqrt{N_g}\left|\frac{\hat\sigma_g-\sigma_g}{\sigma_g}\right|
\lesssim \nu^2\Delta_1.
\]

\paragraph{Step 2: coupling.}
Abbreviate \(X'_t=X_{g_t}(t)-\mu_{g_t}\). Grouping the coupling error by arm gives
\begin{align*}
|\hat H'_T(\lambda,\rho)-H'_T|
&\le \frac{1}{\sqrt T}\sum_{g=1}^k \xi_g,\
\xi_g
&\coloneqq
\left|\frac{1}{\sqrt{N_g}}\sum_{t\in T_g}
\frac{X'_t}{\sigma_g}\right|
\frac{\sqrt{N_g},
|\sigma_g-\hat\sigma_g(\lambda,\rho)|}
{\hat\sigma_g(\lambda,\rho)}.
\end{align*}

\paragraph{Step 3 (\(\rho=\infty\)).}
Here \(\hat\sigma_g(\lambda,\infty)
=\hat\sigma_g+\lambda/\sqrt{N_g}\).
Suppose first that \(\sqrt{N_g}\ge C_2\nu^2\Delta_2\). On \(\mathcal E(\eta)\), \(\hat\sigma_g\ge\sigma_g/2\), and hence
\[
\xi_g
\lesssim
\nu\Delta_2
\left(
\nu^2\Delta_1+\frac{\lambda}{\sigma_g}
\right).
\]
Suppose instead that \(\sqrt{N_g}<C_2\nu^2\Delta_2\). Since
\[
\frac{\sqrt{N_g},
|\sigma_g-\hat\sigma_g(\lambda,\infty)|}
{\hat\sigma_g(\lambda,\infty)}
\le
\sqrt{N_g}+\frac{N_g\sigma_g}{\lambda},
\]
the mean bound on \(\mathcal E(\eta)\) yields
\(
\xi_g
\lesssim
\nu^3\Delta_2^2
+
\nu^5\Delta_2^3\frac{\sigma_g}{\lambda}
\lesssim
\nu^3\Delta_1\Delta_2
+
\nu^5\Delta_1\Delta_2^2\frac{\sigma_g}{\lambda}.
\)
Combining the two cases,
\[
\xi_g
\lesssim
\nu^3\Delta_1\Delta_2
+\nu\Delta_2\frac{\lambda}{\sigma_g}
+\nu^5\Delta_1\Delta_2^2\frac{\sigma_g}{\lambda}.
\]
Summing over \(g\) proves \eqref{eq:simple-reg-coupling}.

\paragraph{Step 4 (\(\rho<\infty\)).}
Put \(\eta_\rho=kT\exp{-\rho/(C\nu^4)}\), with \(C\) sufficiently large. If \(\eta_\rho\ge1\), the claimed probability bound is vacuous. Otherwise, on \(\mathcal E(\eta_\rho)\),
\[
\Delta_1(\eta_\rho)\lesssim \rho/\nu^4,
\qquad
\Delta_2(\eta_\rho)\lesssim \sqrt{\rho}/\nu^2.
\]
If \(N_g>\rho\), no padding is applied. Choosing \(C\) sufficiently large ensures that \(\sqrt{N_g}\ge C_2\nu^2\Delta_2(\eta_\rho)\), so the large-sample bound above gives
\[
\xi_g
\lesssim
\nu^3\Delta_1(\eta_\rho)\Delta_2(\eta_\rho)
\lesssim C_\nu\rho^{3/2}.
\]
If \(N_g\le\rho\), the same deterministic bound as in Step 3 gives
\[
\xi_g
\lesssim
\nu\Delta_2(\eta_\rho)
\left(
\sqrt{\rho}+\rho\frac{\sigma_g}{\lambda}
\right)
\lesssim
C_\nu
\left(
\rho^{3/2}+\rho^2\frac{\sigma_g}{\lambda}
\right).
\]
Thus the same bound holds for every arm. Summing over \(g\) proves \eqref{eq:complex-reg-coupling}.
\end{proof}

\begin{proof}[Proof of Theorems \ref{thm:pooled-pad} and~\ref{thm:pooled-thr}]
    Firstly, put $\delta_t = \{X_{g_t}(t) - \mu_{g_t}\}/\sigma_{g_t}$ and note that by construction we have $\mathbb{E}[\delta_t|\mathcal{F}_{t-1}] = 0$, $\mathbb{E}[\delta_t^2|\mathcal{F}_{t-1}] = 1$,  and $\|\delta_t\|_{\psi_2} \lesssim \nu$. Thus, by a direct application of Corollary \ref{cor:subg-martingale-clt}, we have 
    \begin{equation*}\label{eq:infeasible-clt}
    \sup_{u \in \mathbb{R}}\left|\mathbb{P}\left( H'_T \le u \right) - \Phi(u) \right| \le \frac{C_{\nu} \log T}{\sqrt T}.
    \end{equation*}
   To transfer this to the centered feasible statistic, we note that whenever we have a bound of the form $\mathbb{P}\{|H'_T - \hat H'_T| > \Delta(\eta)\} \le \eta$, Lemma \ref{lemma:combine-coupling-with-kolmogorov} implies that
\begin{equation*}\label{eq:feasible-clt-coupled}
    \sup_{u \in \mathbb{R}}\left|\mathbb{P}\left( \hat H'_T \le u \right) - \Phi(u) \right| \lesssim \frac{C_{\nu} \log T}{\sqrt T} + \eta + \tilde \Delta(\eta) .
    \end{equation*}
   Thus, it suffices to find an appropriate pair of estimates $\eta$ and $\tilde \Delta(\eta)$. This is given by 
   applying the conclusion \eqref{eq:simple-reg-coupling} of Proposition \ref{prop:feasible-coupling} with $\eta = 1/T$ and $\lambda = \
   \sqrt{\log(kT)}$ and $\rho = \infty$, which gives 
\begin{equation*}
    \sup_{u \in \mathbb{R}}\left|\mathbb{P}\left( \hat H'_T \le u \right) - \Phi(u) \right| \lesssim \frac{C_{\nu} \log T}{\sqrt T} + \frac{1}{T} + C'_{\nu}\frac{\log^{3/2}(kT)}{\sqrt T}\left(\sum_{g \le k} \sigma_g \vee \sigma_g^{-1}\right), 
\end{equation*}
which implies the first claimed bound. 
If instead we apply \eqref{eq:complex-reg-coupling} with $\lambda = \infty $, then we obtain 
\begin{equation*}
    \sup_{u \in \mathbb{R}}\left|\mathbb{P}\left( \hat H'_T \le u \right) - \Phi(u) \right| \lesssim \frac{C_{\nu} \log T}{\sqrt T} +  kT\exp\{-\rho/(C\nu^4)\} + C''_{\nu}\frac{k\rho^{3/2}}{\sqrt T}. 
\end{equation*}
Finally, choosing $\rho = C_\nu\log(kT)$ for $C_\nu$ sufficiently large gives the second desired bound.
\end{proof}

\subsection{Proof of Theorem \ref{thm:max-test}}\label{sec:max-proof}

Theorem \ref{thm:max-test} follows an idea sketched in \citet[Section 7]{borovkov1983boundary}. We focus on the following boundary crossing events for Brownian motion.

\begin{lemma}[Robbins-Siegmund Boundary]
Let $\{W(t)\}_{t \ge 0}$ be a standard Wiener process, and suppose $\tau > 0$. Let $\Phi(x)$ denote the standard Gaussian CDF, and put $h(x) = x^2 + 2\log\Phi(x)$. Then:
\begin{align}
    \mathbb{P}\left\{\sup_{t > 1} W(t) \ge \sqrt{t} \cdot h^{-1}\left( \log t+h(a)\right)\right\} &= 1 - \Phi(a) + \Phi'(a)\left[ a + \frac{\Phi'(a)}{\Phi(a)}\right] \label{eq:complex-rs-boundary}; \\
    \mathbb{P}\left\{\sup_{t > \tau} t^{-1}W(t) \ge a\right\} &= 2[1 - \Phi(\tau^{1/2 }a)] \label{eq:simple-rs-boundary}; \\
    \mathbb{P}\left\{\max_{0<t < 1} W(t) \ge a'\right\} &= 2[1 - \Phi(a')]  \label{eq:reflection-principle}
\end{align}
\begin{proof}
Equations \eqref{eq:complex-rs-boundary} and \eqref{eq:simple-rs-boundary} are quoted from \citet[pp.~1411-1412]{robbins1970boundary}, which contains many other one- and two-sided boundary crossing probabilities. Eq.~\eqref{eq:reflection-principle} is the well-known reflection principle (\citealp{cinlar2011brownian}, Prop.~3.4), and is seen to be equivalent to \eqref{eq:simple-rs-boundary} by taking $a = \sqrt{\tau}a'$ and using the time inversion and dilation formulas (respectively, $\{tW(1/t)\}_{t \ge 0} \sim \{W(t)\}_{t \ge 0}$ and $\{\sqrt{\tau}W(t)\}_{t \ge 0} \sim \{W(t/\tau)\}_{t \ge 0}$; \citealp{cinlar2011brownian}, Thm.~1.6).
\end{proof}
\end{lemma}

In order to use \eqref{eq:simple-rs-boundary} and \eqref{eq:reflection-principle} to construct tests in our setting, we use the following quantitative invariance principle due to \citet{sakhanenko1984rate}.
\begin{lemma}[Quantitative invariance principle of \citealp{sakhanenko1984rate}]\label{lem:sakhanenko}
Let $X_1, X_2, \ldots$ be an independent sequence of random variables with $\mathbb{E}[X_i]=0$ and such that for all $i \ge 1$ and some $\lambda > 0$, we have $\mathbb{E}[\lambda|X_i|^3e^{\lambda|X_i|}] \le \mathbb{E}[X^2_i]$. Then, the $X_i$ can be redefined on an enriched probability space along with a sequence of centered Gaussian random variables $Y_1, Y_2, \ldots$ such that $\mathbb{E}[Y_i^2] = \mathbb{E}[X_i^2]$ and
\[\mathbb{E}\exp\left(\lambda C \max_{k \le n} \left|\sum_{i=1}^k X_i - Y_i \right|\right) \le 1 + \lambda\sum_{i=1}^n\mathbb{E}[X_i^2].\]
In particular, if the $X_i$ are i.i.d.~then for $t > 0$,
\[\mathbb{P}\left( \max_{k \le n} \left|\sum_{i=1}^k X_i - Y_i \right| > t\right) \le (1 + \lambda n \mathbb{E}[X_1^2])e^{-\lambda C t}.\]
Finally, if the $X_i$ satisfy Assumption \ref{assn:sub-g} and are standardized to unit variance, then we may take $\lambda \ge 1/C_\nu $, so that
\begin{equation}
    \mathbb{P}\left( \max_{k \le n} \left|\sum_{i=1}^k X_i - Y_i \right| > t\right) \le (1 + n \mathbb{E}[X_1^2])e^{-t / C_\nu}\label{eq:subg-sakhanenko}.
\end{equation}
\begin{proof}
    The first claimed statement is due to Sakhanenko and is cited from \citet[Theorem A]{shao1995strong}. The second follows immediately using Markov's inequality and identical distribution of the $X_i$. 
    
    The third claim follows from the fact that Bernstein's condition, namely $\mathbb{E}[|X_i|^k] \le \tfrac{k!}{2} B^{k-2} \mathbb{E}[X_i^2]$ for $k \ge 3$ and some $B > 0$, implies $\mathbb{E}[\lambda|X_i|^3 e^{\lambda|X_i|}] \le \mathbb{E}[X_i^2]$ for any $\lambda \le B^{-1}/C'$ \citep[p.~732]{zaitsev2013accuracy}; this is checked by writing the Taylor expansion of $e^{\lambda|X_i|}$. To derive Bernstein's from Assumption \ref{assn:sub-g}, apply \Cref{lem:subg-to-bernstein} to $X_i/\sigma$ with $\sigma = \sqrt{\mathbb{E}X_i^2}$, then multiply through by $\sigma^k$ to obtain
    \[
        \mathbb{E}[|X_i|^k] \le \tfrac{k!}{2}\,(C_{\nu/\sigma}\,\sigma)^{k-2}\,\mathbb{E}[X_i^2],
    \]
    so Bernstein's holds with $B = C_{\nu/\sigma}\sigma$. Choosing $\lambda = B^{-1}/C' \wedge 1$ and absorbing constants into $C_\nu$ gives the claimed inequality.
\end{proof}
\end{lemma}

In order to characterize the moderate deviations for our proposed statistic, we need a Gaussian approximation that becomes sharper in the tails. This is accomplished by combining Sakhanenko's coupling with the following argument. 
\begin{lemma}\label{lem:non-uniform-approximation-from-coupling}
    Suppose that $U$ and $V$ are random variables, and that $V$ has probability density function $f$. Suppose also that $\mathbb{P}\{U > V +\delta\} \le \eta$. Then,
    \[\mathbb{P}\{U > t\} - \mathbb{P}\{V > t\}  \le  \eta + \mathbb{P}\{V \in (t-\delta, t]\} =  \eta + \int_{t-\delta}^t f(v)\,dv.\]
    \begin{proof}
        Clearly the difference on the left is bounded by $\mathbb{P}\{V  \le t < U \}$. On this event, either $U > V + \delta$ or $V \in (t-\delta, t]$. We conclude by a union bound and write $\mathbb{P}\{V \in (t-\delta, t]\}$ as an integral.
    \end{proof}
\end{lemma}

Finally, we quote the Dudley-Strassen theorem \citep[Theorem 2]{dudley1968distances}, which we use to combine probabilistic approximations.

\begin{lemma} \label{lem:dudley-strassen}
    Let $U, V$ be random variables taking values in a complete, separable metric space $(S,d)$. Let for $x \in S$, let $\bar B_\delta(x) = \set{y \in S:d(x,y) \le \delta}$ denote the closed ball of radius $\delta$ centered at $x$, and define $A^\delta = \cup_{a \in A} \bar B_\delta(a)$ to be the $\delta$-enlargement of the set $A \subset S$ Then the following are equivalent.
    \begin{enumerate}
        \item For all closed sets $A \subset S$, $\mathbb{P}(U \in A) \le \mathbb{P}(V \in A^\delta) + \eta$.
        \item There exists a Borel probability measure $\mathbb{P}'$ on $S\times[0,1]$ and random variables $U'$, $V'$ such that $\mathbb{P}'\{d(U',V')>\delta\} \le \eta$.
    \end{enumerate}
\end{lemma}

\subsubsection{Coupling Lemmas}
To simplify notation, we write
\[\hat \sigma^2_g(q) = \frac{1}{q}\sum_{t \in T_{g,q}} [X_{g_s}(s) - \hat\mu_{g,q}]^2\]
for the standard variance estimator applied to the first $q$ draws of arm $g$.
We begin by defining (infeasible) centered and standardized versions of the statistic $\hat Z_g(q)$ from \eqref{eq:def-arm-t-stat}:
\[\hat Z'_g(q) = \sum_{s\in T_{g,q} }\frac{ X_{g_s}(s) - \mu_{g}}{\sqrt{q}\hat\sigma_g(q)}; \quad Z_g(q) = \sum_{s\in T_{g,q} } \frac{ X_{g_s}(s) - \mu_{g}}{\sqrt{q}\sigma_g}.\]
We begin by showing that $\hat Z'_g$ and $Z_g$ are close in a time-uniform sense.

\begin{lemma}
    Given $\eta \in (0,1)$, suppose that $(\nu \vee\nu^2) \sqrt{k\log(T/\eta)/T} \le 1/3$. Then, with probability at least $1-\eta$, it holds for all arms $1 \le g \le k$ that 
    \begin{equation}
        \label{eq:infeasible-t-stat-to-feasible}
    \max_{T/k \le q \le T} \left|\hat Z'_g(q)-Z_g(q) \right| \lesssim \nu^3\left(\frac{\log^{3/2}(kT/\eta)}{\sqrt{T/k}}\right).
    \end{equation}
\begin{proof}
Note that the bounds of Lemmas \ref{lemma:normalized-bound-mean} and \ref{lemma:normalized-bound-var} hold not only at the random time $N_g$ but at all times, $q$. Combining this observation with a straightforward modification of the proof of Lemma \ref{lem:consistent-estimator-variance} yields that on an event with probability $1-\eta$,
\begin{align} 
\label{eq:mean-consistency-time-unif}
\max_{T/k \le q \le T}  |Z_g(q)| &\lesssim \nu \sqrt{{\log(T/\eta)}} 
\\
\label{eq:var-consistency-time-unif}
\max_{T/k \le q \le T}  \left|\frac{\hat\sigma_g(q) - \sigma_g}{\sigma_g}\right| &\lesssim \nu^2 \left( \frac{\log(T/\eta)}{T/k} + \frac{\sqrt{\log(T/\eta)} + \log(1/\eta)}{\sqrt{T/k}} \right) 
\intertext{and}
\label{eq:var-lb-time-unif}
\max_{T/k \le q \le T} \frac{\sigma_g - \hat \sigma_g(q)}{\sigma_g} &\lesssim \nu^2 \left( \frac{\log(T/\eta)}{T/k} + \sqrt{\frac{\log(T/\eta)}{T/k} }\right) 
\end{align}
Writing $X'_g(s) = X_g(s) - \mu_g$, we then have on the same event that
\begin{align*}
    \max_{T/k \le q \le T}  \left\{\hat Z'_g(q)- Z_g(q) \right\} 
    &=\max_{T/k \le q \le T} \left| \left(\frac{\hat\sigma_g(q) - \sigma_g}{\hat\sigma_g(q)}\right)\sum_{s\in T_{g,q} }\frac{ X_{g_s}'(s)}{\sqrt{q}\sigma_g}\right|  \\
    &\le \max_{T/k \le q \le T} \left| \frac{\hat\sigma_g(q) - \sigma_g}{\hat\sigma_g(q)}\right| \max_{T/k \le q \le T} \left|\sum_{s\in T_{g,q} }\frac{ X_{g_s}'(s)}{\sqrt{q}\sigma_g}\right|  
    \intertext{Since combining our assumption $(\nu \vee\nu^2) \sqrt{k\log(T/\eta)/T} \le 1/3$ with \eqref{eq:var-lb-time-unif} implies $\hat\sigma_g(q) \ge \sigma_g/2$ for all $T/k \le q \le T$, we may apply the bounds \eqref{eq:mean-consistency-time-unif} and \eqref{eq:var-consistency-time-unif} to obtain}
    &\lesssim \nu^3 \sqrt{\log(T/\eta)} \left( \frac{\log(T/\eta)}{T/k} + \frac{\sqrt{\log(T/\eta)} + \log(1/\eta)}{\sqrt{T/k}} \right) \\
    &\lesssim \nu^3\left(\frac{\log^{3/2}(T/\eta)}{\sqrt{T/k}}\right). 
\end{align*}
The result then folows by a union bound over arms $1 \le g \le k$.
\end{proof}
\end{lemma}
Next, we apply Sakhanenko's invariance principle to show that each sequence $Z_g(t)$ is uniformly close to a sequence of Gaussian partial sums. 
\begin{lemma}
Under Assumption \ref{assn:sub-g}, for each $1 \le g \le k$ we may construct a probability space with random variables $\{\alt Z_g(t)\}_{t \ge 1} \sim \{Z_g(t)\}_{t \ge 1}$ and independent standard normal random variables $\{\alt\, Y_{g,t}\}_{t \ge 1}$ such that with probability $1-\eta$,
\begin{equation}\max_{T/k \le q \le T} \left| \sqrt{q}\,\alt Z_g(q) -  \sum_{i=1}^q \alt \, Y_{g,i}\right| \le C_\nu\log\left(\frac{1 +T}{\eta}\right).
\label{eq:infeasible-t-stat-to-gaussian}
\end{equation}
\begin{proof}
This follows immediately by inverting the bound \eqref{eq:subg-sakhanenko} from Lemma \ref{lem:sakhanenko}.
\end{proof}
\end{lemma}

Finally, we apply the Dudley-Strassen theorem to consolidate the previous two results. 

\begin{lemma} \label{lem:t-stat-brownian-embedding}
Under Assumption \ref{assn:sub-g}, for each $1 \le g \le k$ we may construct a probability space with random variables $\{\alttwo \hat Z'_g(t)\}_{t \ge 1} \sim \{\hat Z'_g(t)\}_{t \ge 1}$ and a standard Wiener process $\{\alttwo\,W_g(t)\}_{t \ge 0}$ such that with probability $1-\eta$,
\[\max_{T/k \le q \le T} \left|  \alttwo \hat Z'_g(q) -  \frac{{\alttwo\,W}_g(q)}{\sqrt q} \right| \lesssim  C_\nu\left( \frac{\log^{3/2}(kT/\eta)}{\sqrt{T/k}}\right) .\]
\begin{proof}
We repeatedly apply the Dudley-Strassen theorem (Lemma \ref{lem:dudley-strassen}) in the complete, separable metric space $\ell^\infty(J)$ for $J = \{T/k, T/k+1, \ldots, T\}$, writing $\|x-y\|_J = \max_{T/k \le q \le T} |x_q-y_q|$ for the resulting metric. 

In particular, suppose we are given a closed subset $A \subset \ell^\infty(J)$  and $\eta' \in (0,1)$. Then, taking $\eta = \eta'/2$ in \eqref{eq:infeasible-t-stat-to-feasible} and using Lemma \ref{lem:dudley-strassen} gives 
\[\mathbb{P}\left\{Z_g(q) \in A^{\epsilon}\right\} \le \mathbb{P}\left\{\hat Z'_g(q)  \in (A^{\epsilon})^{\delta} \right\} + \eta'/2 =  \mathbb{P}\left\{\hat Z'_g(q)  \in A^{\epsilon+\delta} \right\} + \eta'/2 ,\]
for $\delta = C_\nu\sqrt{k\log^3(2kT/\eta)/T}$.

Then, taking $\eta = \eta'/2$ in \eqref{eq:infeasible-t-stat-to-gaussian} and using Lemma \ref{lem:dudley-strassen} similarly gives that for a possibly different probability space $\mathbb{P}^\sharp$, 
\begin{align*}
    \mathbb{P}^\sharp\left\{\frac{1}{\sqrt{q}} \sum_{i=1}^q \alt\,Y_{g,i} \in A \right\}
    &\le \mathbb{P}^\sharp\left\{\alt Z_g(q) \in A^\epsilon \right\} + \eta'/2,
\intertext{for $\epsilon = C_\nu\sqrt{k\log^2\{(1+T)/\eta\}/T}$. Using the fact that $\alt Z_g$ and $Z_g$ are equally distributed and applying the bound of the preceding display, this is at most}
    &\le \mathbb{P}\left\{\hat Z'_g(q) \in A^{\epsilon+\delta} \right\} + \eta'.
\end{align*}
Finally, by applying Lemma \ref{lem:dudley-strassen} in the reverse direction, we may construct on a third probability space $\mathbb{P}^\flat$ random variables $\{\alttwo \hat Z'_g(t)\}_{t \ge 1} \sim \{\hat Z'_g(t)\}_{t \ge 1}$ and and independent standard normal random variables $\{\alttwo\, Y_{g,t}\}_{t \ge 1}$  such that with probability $1-\eta'$,
\[\max_{T/k \le q \le T} \left| \alttwo \hat Z'_g(q) -  \frac{1}{\sqrt{q}} \sum_{i=1}^q \alttwo\,Y_{g_i}\right| \le \delta + \epsilon.\]
Finally, by standard generalities (namely, Kolmogorov's extension and continuity theorems), we may extend the construction so that $ \sum_{i=1}^q \alttwo\,Y_{g_i} = {\alttwo\,W}_g(q)$ almost surely for a standard Wiener process ${\alttwo\,W}_g(q)$. We conclude by absorbing constants in the definition of $\delta$ and $\epsilon$, using our maintained assumption that $T \ge k \ge 2$.
\end{proof}
\end{lemma}

\subsubsection{Proofs of main results}

Using the coupling of Lemma \ref{lem:t-stat-brownian-embedding}, we can prove a moderate deviation principle for the linear boundary statistic $\max_{T/k \le q \le T} \sqrt{T/(kq)}\,\hat Z'_g(q)$, and its logarithmic boundary counterpart. In the linear boundary case, we may show that for suitable sequences $c_T, k_T \uparrow \infty$ as $T \uparrow \infty$,
\[[1 - o_p(1)][2-2\Phi(c_T)] \le \mathbb{P}\left\{\max_{T/k_T \le q \le T} \sqrt{\frac{T/k_T}{q}}\,\hat Z'_g(q) > c_T\right\}\le [1 + o_p(1)][2-2\Phi(c_T)].\] 
Crucially, the above approximation allows a union bound over a potentially large number of arms $k_T \uparrow \infty$, therefore it strictly generalizes the approximations provided by \citet{waudbysmith2023distribution,waudbysmith2024time}, who in turn build upon \citet{robbins1970boundary}.

\subsubsection*{Linear boundary}

Here we focus on the upper bound, which is all that is needed to prove Theorem \ref{thm:max-test}. In particular, we show the following. 

\begin{proposition}[Moderate deviations upper bound for linear boundary] \label{prop:mdp-lin} Given sequences $\{r_T\}_{T \ge 1}$ and $\{k_T\}_{T \ge 1}$, and any start time $q_0$ with $T/k_T \le q_0 \le T$, put
\[\psi_T^2 = \frac{[r_T^2\vee 1]\log^3(k_TT) + r_T^8}{T/k_T}. \]
Then, there exists a universal constant $C > 0$ such that
\[\mathbb{P}\left\{\max_{q_0 \le q \le T} \sqrt{\frac{q_0}{q}}\,\hat Z'_g(q) > r_T\right\}\le [1 + C(T^{-1} + \psi_T\,e^{C\psi_T})][2-2\Phi(r_T)].\]
In particular, the right-hand side is bounded as $[1 + o(1)][2-2\Phi(r_T)]$ whenever $\psi_T \downarrow 0$.
\begin{proof}
We work in the probability space $\mathbb{P}^\flat$ from Lemma \ref{lem:t-stat-brownian-embedding}, and omit the $\flat$ symbol to simplify notation. We begin by noting that by the time inversion formula $\{tW(1/t)\}_{t \ge 0} \sim \{W(t)\}_{t \ge 0}$ (\citealp{cinlar2011brownian}, Thm.~1.6) and Brownian scaling,
\begin{align*}
    \max_{q_0 \le q \le T} \frac{\sqrt{q_0}}{q} W_g(q) \le \sup_{u \in [q_0,\infty)} \frac{\sqrt{q_0}}{u} W_g(u)
    &\stackrel{d}{=} \sup_{u \in (0,1/q_0]} \sqrt{q_0}\,W_g(u) \stackrel{d}{=} \sup_{u \in (0,1]} W_g(u),
\end{align*}
where the final equality in distribution is crucially independent of $q_0$. Therefore, by Lemma \ref{lem:t-stat-brownian-embedding} (whose coupling, stated on $[T/k,T]$, applies a fortiori on the subinterval $[q_0,T]$), the triangle inequality, and $q_0 \le q$, we have
\[\mathbb{P}\left\{\max_{q_0 \le q \le T}  \sqrt{\frac{q_0}{q}}\, \hat Z'_g(q) - \sup_{u \in (0,1]} W_g(u) > \delta(\eta) \right\} \le \eta,\]
for $\delta(\eta) = C_\nu\sqrt{k\log^3(2kT/\eta)/T}$.
By Lemma \ref{eq:simple-rs-boundary} and Brownian scaling, for any $r > 0$,
\[\mathbb{P}\left\{ \sup_{u \in (0,1]} W_g(u) > r \right\} = 2[1-\Phi(r)].\]
Combining the previous two displays using Lemma \ref{lem:non-uniform-approximation-from-coupling}, we obtain for any $\eta \in(0,1)$
 \[\mathbb{P}\left\{\max_{q_0 \le q \le T}  \sqrt{\frac{q_0}{q}}\, \hat Z'_g(q) > r\right\} - 2[1-\Phi(r)] \le  \eta + 2\int_{r-\delta(\eta)}^r \varphi(v)\,dv.\]
 We first treat the case in which $r \le 1$. Then $1 - \Phi(r) \gtrsim 1$ and we may bound
 \[\eta + 2\int_{r-\delta(\eta)}^r \varphi(v)\,dv \lesssim \eta + \delta(\eta)\] by boundedness of the Gaussian density. Choosing $\eta = 1/T$, we obtain the bound
  \[\mathbb{P}\left\{\max_{q_0 \le q \le T}  \sqrt{\frac{q_0}{q}}\, \hat Z'_g(q) > r\right\} - 2[1-\Phi(r)] \le  C_\nu'\sqrt{\frac{\log^3(kT)}{T/k}} \cdot [1-\Phi(r)].\]
  On the other hand, if $r \ge 1$ then we have $1- \Phi(r) \gtrsim r^{-1}e^{-r^2/2}$ by Mill's inequality. Choosing $\eta = T^{-1}r^{-1}e^{-r^2/2} \lesssim T^{-1}2[1-\Phi(r)]$, we have
  \[\delta(\eta) = C_\nu\sqrt{\frac{\log^3(2kT/\eta)}{T/k}}\le C_\nu''\left\{ \sqrt{\frac{\log^3(kT) + r^6}{T/k}}\right\}.\]
  using the inequalities $(x + y)^3 \lesssim x^3 + y^3$ and $\log r \le r$ for $r \ge 1$. 
  Moreover, since $\vp(-)$ is decreasing, 
  \begin{align*}
  \int_{r-\delta(\eta)}^r \varphi(v)\,dv \le \delta(\eta)\vp\{r-\delta(\eta)\} 
  &\lesssim \delta(\eta)\exp\left\{\frac{-[r - \delta(\eta)]^2}{2}\right\} \\
  &\le \delta(\eta)\exp\left\{\frac{-r^2 + 2r\delta(\eta)}{2}\right\} \\
  &\lesssim [1-\Phi(r)][r\delta(\eta)]e^{r\delta(\eta)}.
  \end{align*}
We recover the sufficient rate condition $r\delta(\eta) \downarrow 0$, which amounts to
\[ \frac{r^2\log^3(kT) + r^8}{T/k} \downarrow 0.\]
By inspection, this also covers the $r \le 1$ case if we replace $r^2$ by $r^2\vee1$, leading to the claimed bound.
\end{proof}
\end{proposition}

\subsubsection*{Logarithmic Boundary}
The case of the logarithmic boundary requires a few more technical details due to the complexity of the functions
\[h(x) =x^2+2\log\Phi(x), \qquad \Psi_+(x) = 1 - \Phi(x) + \Phi'(x)\left[ x + \frac{\Phi'(x)}{\Phi(x)}\right].\]
To keep the exposition simple, we collect the essential facts about these functions in two lemmas.
\begin{lemma}\label{lem:boundary-transform} The function $h(x)$ satisfies $h(x) \sim x^2$ and $h'(x) \sim 2x$. It is a monotone, continuously differentiable bijection $[x_0, \infty) \to [x_0,\infty)$ for some fixed point $x_0 \in (1,\sqrt{2})$, and its inverse is monotone and continuously differentiable.  Moreover, (i) for all $x \ge x_0$, $y \ge 0$,
\(\left.\frac{d}{du}\right|_{u=h(x)+y} h^{-1}(u) \le \frac{1}{2x}\), and (ii) $h^{-1}(h(x)-y)$ is $O(1)$-Lipschitz in $x$ for $x \ge \sqrt{2y} \vee x_0$ and $y \ge 0$.
\end{lemma}
\begin{proof}
    The first claim follows since $\Phi(x) \to 1$ as $x \uparrow \infty$. 
    To further characterize $h$, we compute its derivative
    \[h'(x) = 2\left[x + \frac{\Phi'(x)}{\Phi(x)}\right].\] Clearly $2x \le h'(x) \le 2x + e^{-x^2/2}$, and $h$ is strictly increasing and continuous for $x \ge 1$. Since $h(1) < 1$ and $h(\sqrt{2}) > \sqrt{2}$, it follows by the intermediate value theorem that $h$ has a fixed point $x_0 \in (1,\sqrt{2})$, and both $h$ and its inverse $h^{-1}$ are well defined as functions $[x_0, \infty) \to [x_0,\infty)$.

    Next, for $x \ge x_0$, the inverse function theorem gives
    \(\left.\frac{d}{du}\right|_{h(x)} h^{-1}(u) = \frac{1}{h'(x)} \le \frac{1}{2x}.\) From this, and the implied monotonicity of $h^{-1}$, we deduce our third claim that for any $y \ge 0$,
    \[\left.\frac{d}{du}\right|_{u=h(x)+y} h^{-1}(u) = \frac{1}{h'\{h^{-1}[h(x) + y]\}} \le \frac{1}{2h^{-1}[h(x) + y] } \le \frac{1}{2x} .\]
    Finally, suppose that $y \ge 0$ and also that $x \ge \sqrt{2y} \vee x_0$. Then we have 
    \[h(x)-y \ge h(x) - x^2/2 =  x^2/2 + 2\log\Phi(x) \ge x^2/2 + 2\log\Phi(x/\sqrt{2}) = h(x/\sqrt{2}).\]
    Thus,
    \[\frac{\partial}{\partial x} h^{-1}(h(x)-y) = \frac{h'(x)}{h'\{h^{-1}[h(x) - y]\}} \le \frac{\sqrt{8}x + e^{-x^2/2}}{2h^{-1}[h(x) - y] } \le \frac{\sqrt{8}x + e^{-x^2/2}}{\sqrt{2}x } \lesssim 1.\]
    This proves our final claim.
\end{proof}
\begin{lemma}\label{lemma:psi-characterization}
For all $x \ge 1$ we have $\Psi_+(x) \sim x\Phi'(x)$ and $-\Psi_+'(x) \sim x^2\Phi'(x)$:
\begin{align*}
    (1/2x + x)\Phi'(x) &\le \Psi_+(x) \le (1/x + x + e^{-x^2/2})\Phi'(x), \\
    \quad x^2\Phi'(x) &\le -\Psi'_+(x) \le [x^2 + xe^{-x^2/2} + e^{-x^2}]\Phi'(x).
\end{align*}
In particular, if $\Psi_+(w_\alpha(k)) = \alpha/k$ for some $\alpha \in (0,1)$ and $k \ge 2$ then \[\sqrt{2\log k} \le w_\alpha(k) \le C \sqrt{\log (k/\alpha)}.\]
\end{lemma}
\begin{proof}
   Note that by the Mill's ratio inequality we have for all $x \ge 1$ that
   \[\frac{e^{-x^2/2}}{6x} \le \frac{\Phi'(x)}{2x} \le 1-\Phi(x) \le \frac{\Phi'(x)}{x} \le \frac{e^{-x^2/2}}{2x}. \]
   Thus, for $x \ge 1$, since $\Phi'(x) \le \frac{1}{2}e^{-x^2/2}$ and $1/2 \le \Phi(x) \le 1$, we have
   \[ \left(\frac{1}{2x} + x \right)\Phi'(x) \le 1 - \Phi(x) + x\Phi'(x) + \Phi'(x)^2/\Phi(x) \le \left(\frac{1}{x} + x + e^{-x^2/2}\right)\Phi'(x).\]
    Since the quantity in the middle is $\Psi_+(x)$, the bounds follow immediately. 

    The bounds on critical values follow: for $k \ge 2$, $\sqrt{2\log k} \ge 1$, and plugging in to the lower bound gives a value larger than $1/k$. For the upper bound, it is immediate that our upper bound on $\Psi_+$ is at most $e^{-c x^2}$ some small $c > 0$, so $w_\alpha(k) \lesssim \sqrt{\log(k/\alpha)}$.
    
    To bound $-\Psi_+'$, the observation $\Phi''(x) = -x\Phi'(x)$ implies
    \[-\frac{d}{dx} \left[ 1 - \Phi(x) + x\Phi'(x) + \frac{\Phi'(x)^2}{\Phi(x)} \right]=  x^2\Phi'(x) + \frac{2x\Phi'(x)^2}{\Phi(x)} + \frac{\Phi'(x)^3}{\Phi(x)^2}.\]
    We conclude again by repeatedly applying $\Phi'(x) \le \frac{1}{2}e^{-x^2/2}$ and $1/2 \le \Phi(x) \le 1$. 
\end{proof}
\begin{proposition}[Moderate deviations upper bound for logarithmic boundary] \label{prop:mdp-log} Given sequences $\{r_T\}_{T \ge 1}$ and $\{k_T\}_{T \ge 1}$ with $r_T \ge \sqrt{2\log k_T}$, and any start time $q_0$ with $T/k_T \le q_0 \le T$, put
\[\xi_T^2 = \frac{r_T^2\log^3(k_TT) + r_T^8}{T/k_T}. \]
Then, there exists a universal constant $C > 0$ such that
\[\mathbb{P}\left(\max_{q_0 \le q \le T}\hat Z'_g(q)- h^{-1}[\log(q/q_0) + h(r_T)] >0\right) \le [1 + C(T^{-1} + \xi_T e^{\xi_T})]\Psi_+(r_T)\]
In particular, the right-hand side is bounded as $[1 + o(1)]\Psi_+(r_T)$ whenever $\xi_T \downarrow 0$.

\end{proposition}
\begin{proof}
We again work in the probability space $\mathbb{P}^\flat$ from Lemma \ref{lem:t-stat-brownian-embedding}, and omit the $\flat$ symbol to simplify notation. Let $r > \sqrt{2\log k} \eqqcolon \underline{r}$ be given and define $l_{\underline r}(x) = x \vee \underline{r}$. Lemma \ref{lem:t-stat-brownian-embedding} (which applies a fortiori on the subinterval $[q_0,T] \subset [T/k,T]$), along with the fact that $l_{\underline r}$ is $1$-Lipschitz, gives that w.p.~$1-\eta$,
\[\max_{q_0 \le q \le T} \left| l_{\underline{r}}(\hat Z'_g(q)) -  l_{\underline{r}}\left(\frac{{\,W}_g(q)}{\sqrt q}\right) \right|  \le \max_{q_0 \le q \le T} \left|   \hat Z'_g(q) -  \frac{{\,W}_g(q)}{\sqrt q} \right| \le  C_\nu\left( \frac{\log^{3/2}(kT/\eta)}{\sqrt{T/k}}\right) .\] Note that the mapping $x \mapsto u(x,y) = h^{-1}(h(x) - y)$ is $O(1)$-Lipschitz for $x \ge \sqrt{2y}$ (Lemma \ref{lem:boundary-transform}) and $q / q_0 \le T/q_0 \le k$ for $q_0 \le q \le T$, so that we have ensured \(l_{\underline r}(-) \ge \underline{r} = \sqrt{2\log k} \ge \sqrt{2\log(q/q_0)}.\) Thus, it holds w.p.~$1-\eta$,
\begin{align*}
    & \left| \max_{q_0 \le q \le T}u\{l_{\underline{r}}(\hat Z'_g(q)),\log(q/q_0)\} - \max_{q_0 \le q \le T} u\left\{l_{\underline{r}}\left(\frac{{\,W}_g(q)}{\sqrt q}\right), \log(q/q_0)\right\} \right| \\
    & \qquad \le \max_{q_0 \le q \le T} \left| u\{l_{\underline{r}}(\hat Z'_g(q)),\log(q/q_0)\} -  u\left\{l_{\underline{r}}\left(\frac{{\,W}_g(q)}{\sqrt q}\right), \log(q/q_0)\right\} \right|  \\
    & \qquad \le C\max_{q_0 \le q \le T} \left| l_{\underline{r}}(\hat Z'_g(q)) -  l_{\underline{r}}\left(\frac{{\,W}_g(q)}{\sqrt q}\right) \right|\le  \delta(\eta),
\end{align*}
for $\delta(\eta) = C_\nu'\left( \frac{\log^{3/2}(kT/\eta)}{\sqrt{T/k}}\right)$, where the last inequality follows from the preceding display. This further implies that w.p.~$1-\eta$,
\begin{align*}
    \max_{q_0 \le q \le T}u\{l_{\underline{r}}(\hat Z'_g(q)),\log(q/q_0)\}
    &\le \max_{q_0 \le q \le T} u\left\{l_{\underline{r}}\left(\frac{{\,W}_g(q)}{\sqrt q}\right), \log(q/q_0)\right\} + \delta(\eta) \\
    &\le \sup_{t > q_0 } u\left\{l_{\underline{r}}\left(\frac{{\,W}_g(t)}{\sqrt t}\right), \log(t/q_0)\right\} + \delta(\eta),
\end{align*}
where in the last step we relaxed the maximum over integers $q_0 \le q \le T$ to be over real numbers $t > q_0$. Following Lemma \ref{lem:non-uniform-approximation-from-coupling}, we have
\begin{equation}\label{eq:log-nonuniform-comparison}
\begin{split}
    \mathbb{P}\left(\max_{q_0 \le q \le T}u\{l_{\underline{r}}(\hat Z'_g(q)),\log(q/q_0)\}  > r\right) &\le \mathbb{P}\left(\sup_{t > q_0 } u\left\{l_{\underline{r}}\left(\frac{{\,W}_g(t)}{\sqrt t}\right), \log(t/q_0)\right\} > r\right)  + \eta \\ &  + \mathbb{P}\left(\sup_{t > q_0 } u\left\{l_{\underline{r}}\left(\frac{{\,W}_g(t)}{\sqrt t}\right), \log(t/q_0)\right\} \in (r - \delta(\eta), r]\right)
    \end{split}
\end{equation}

Next, we show that for any $x,y > 0$, we have 
\( u(l_{\underline r}(x),y) \le r \iff u(x,y ) \le r.\)
Note that we have chosen $\underline r < r$ and that $u$ is increasing in its first argument, decreasing in its second argument, and satisfies $ u(x,0) = x$ for any $x$. Thus $u(l_{\underline r}(x),y) \le r$ immediately implies $u(x,y) \le r$. Conversely, suppose that $u(x,y) \le r$. Then, either (i) $x \ge \underline{r}$ in which case $u(l_{\underline r}(x),y) = u(x,y) \le r$ or else (ii) $x < \underline{r}$ so that $l_{\underline r}(x) = \underline r$ and $u(l_{\underline r}(x),y) = u( \underline{r},y) \le u(\underline r,0) = \underline r < r$.

Using these observations along with the definition $u(x,y) = h^{-1}[h(x)-y]$ we have for any sequence $\{x_q\}_{q \ge 1}$
\begin{equation}
\label{eq:truncation}
\begin{split}
   \max_{q_0 \le q \le T} u\left\{l_{\underline{r}}(x_q), \log(\nicefrac{q}{q_0})\right\} > r
    &\iff \max_{q_0 \le q \le T} u\left\{x_q, \log(\nicefrac{q}{q_0})\right\} > r  \\
     &\iff \max_{q_0 \le q \le T} x_q - h^{-1}[\log(\nicefrac{q}{q_0}) + h(r)] > 0
\end{split}
\end{equation}
Applying \eqref{eq:truncation},
\begin{align*}
    \mathbb{P}\left(\sup_{t > q_0} u\left\{l_{\underline{r}}\left(\frac{{\,W}_g(t)}{\sqrt t}\right), \log(\nicefrac{t}{q_0})\right\} > r \right)
    &= \mathbb{P}\left(\sup_{t > q_0} \frac{{W}_g(t)}{\sqrt t} - h^{-1}[\log(\nicefrac{t}{q_0}) + h(r)] > 0\right).
\intertext{Using the scaling property $\{W_{g}(at)\}_{t \ge 0} \sim \{\sqrt{a}W(t)\}_{t \ge 0}$ with $a = q_0^{-1}$ then gives}
&= \mathbb{P}\left(\sup_{t > q_0} \frac{{W}_g((\nicefrac{t}{q_0}))}{\sqrt {(\nicefrac{t}{q_0})}}- h^{-1}[\log(\nicefrac{t}{q_0}) + h(r)] > 0\right) \\
&= \mathbb{P}\left(\sup_{s >1 } \frac{{W}_g(s)}{\sqrt {s}} - h^{-1}[\log(s) + h(r)] > 0 \right).
\end{align*}
where the last line takes $s = q_0^{-1}t$ and is crucially independent of $q_0$. Applying \eqref{eq:complex-rs-boundary}, we finally get
\begin{equation}
     \mathbb{P}\left(\sup_{t > q_0} u\left\{l_{\underline{r}}\left(\frac{{\,W}_g(t)}{\sqrt t}\right), \log(\nicefrac{t}{q_0})\right\} > r \right) = \Psi_+(r). \label{eq:wiener-remove-truncation}
\end{equation}
Using \eqref{eq:truncation} again, we similarly obtain
\begin{equation}
   \mathbb{P}\left(\max_{q_0 \le q \le T}u\{l_{\underline{r}}(\hat Z'_g(q)),\log(q/q_0)\}  > r\right) = \mathbb{P}\left(\max_{q_0 \le q \le T}\hat Z'_g(q)- h^{-1}[\log(\nicefrac{q}{q_0}) + h(r)] > 0 \right) \label{eq:tstat-remove-truncation}
\end{equation}
Plugging \eqref{eq:wiener-remove-truncation} and \eqref{eq:tstat-remove-truncation} into \eqref{eq:log-nonuniform-comparison} and using the fundamental theorem of calculus, we get
\begin{align*} \mathbb{P}\left(\max_{q_0 \le q \le T}\hat Z'_g(q)- h^{-1}[\log\nicefrac{q}{q_0} + h(r)] > 0\right) - \Psi_+(r)
&\le \eta + \int_{r-\delta(\eta)}^r - \Psi_+'(s)\,ds \\
&\le \eta + C\delta(\eta)r^2e^{-[r-\delta(\eta)]^2/2},
\end{align*}
where we have used Lemma \ref{lemma:psi-characterization} to bound $-\Psi_+'(-)$ on the right-hand side. Choosing $\eta = T^{-1}e^{-r^2/2} \lesssim T^{-1}\Psi_+(r)$ we have
\[\delta(\eta) \le C_\nu\sqrt{\frac{\log^3(2kT/\eta)}{T/k}}\lesssim C_\nu\left\{ \sqrt{\frac{\log^3(kT) + r^6}{T/k}}\right\}.\]
We then bound
\[\delta r^2\exp\{-[r-\delta]^2/2\} \le \delta r^2e^{-r^2/2 + \delta r} \lesssim \delta r e^{\delta r} \Psi_+(r),\] again using the bound of Lemma \ref{lemma:psi-characterization}. This yields the final bound 
\[\mathbb{P}\left(\max_{q_0 \le q \le T}\hat Z'_g(q)- h^{-1}[\log\nicefrac{q}{q_0} + h(r)] > 0\right)  - \Psi_+(r)\lesssim [T^{-1} + \xi e^\xi]\Psi_+(r)\]
for
\( \delta(\eta)r \lesssim C_\nu\left\{ \sqrt{\frac{r^2\log^3(kT) + r^8}{T/k}}\right\} \eqqcolon \xi,\) as claimed.
\end{proof}

\subsubsection*{Proof of Theorem \ref{thm:max-test}} For the first claim, note that under the null $\mu_g \le 0$, the centered statistic $\hat Z'_g(q)$ considered by Propositions \ref{prop:mdp-lin} and \ref{prop:mdp-log} almost surely exceeds  $\hat Z_g(q)$:
\[
\hat Z_g(q)
=
\hat Z'_g(q) + \frac{\sqrt q\,\mu_g}{\hat\sigma_g(q)}
\le
\hat Z'_g(q).
\] Recall that the tests in Theorem \ref{thm:max-test} are indexed by a parameter $\zeta \ge 1$, with the max taken over arms $g \in \mathfrak{K}(t,\zeta)$, i.e., $N_g(t) \ge \zeta T/k$. We therefore instantiate Propositions \ref{prop:mdp-lin} and \ref{prop:mdp-log} with start time $q_0 = \zeta T/k$; since $\zeta \ge 1$, the condition $T/k_T \le q_0 \le T$ holds for all $T$ large enough. Moreover, by Lemma \ref{lemma:psi-characterization} we have $\sqrt{2\log k} \le w_\alpha(k_T) \lesssim \sqrt{\log{k}}$. Therefore, taking $r_T = w_\alpha(k_T)$ in the setting of Proposition \ref{prop:mdp-log}, it is clear that \[T/[k_T\log^{4}(Tk_T)] \uparrow \infty \implies \xi_T^2 = \frac{r_T^2\log^3(k_TT) + r_T^8}{T/k_T} \to 0,\]
which implies by a union bound that under the null,
\[\mathbb{E}(A_{\textrm{log}}) \le \sum_{g = 1}^k \frac{\alpha}{k}[1+o(1)] = [1 + o(1)]\alpha.\]
The $1 - \alpha/(2k)$ quantile of the standard normal distribution, which is $z_\alpha(k)$, is known to be $O(\sqrt{\log\nicefrac{k}{\alpha}})$; this follows, e.g.,~from the Mill's ratio inequality $1 - \Phi(x) \lesssim x^{-1}e^{-x^2/2}$. Thus, Proposition \ref{prop:mdp-lin} can similarly be applied (with the same $q_0 = \zeta T/k$) to deduce that
\(\mathbb{E}(A_{\textrm{lin}}) \le [1 + o(1)]\alpha.\)

For claim (ii), note that 
\begin{align*}
    h^{-1}\{h[w_{\alpha}(k_T)] + \log(qk/T)\} &\le h^{-1}\{h[w_{\alpha}(k_T)] + \log k\} \\ &= w_{\alpha}(k_T) + \int_{h[w_\alpha(k_T)]}^{h[w_\alpha(k_T)]+\log k}  [h^{-1}]'(s)\,ds.
\end{align*}
By Lemma \ref{lem:boundary-transform}, since $w_\alpha(k_T) \ge \sqrt{2\log k_T}$, the integrand is bounded by $(\log k_T)^{-1/2}$ uniformly over the domain of the integral. Thus the integral, and hence the left hand side, is also bounded by $\sqrt{\log k_T}$, as claimed.
\section{Section 4 Proofs}
\subsection{Confidence bound}

The goal of this subsection is to prove the following two-sided confidence bound for $\mu_g/\sigma_g$. 
\begin{proposition}\label{prop:self-normalized-cb}
For any $\beta > 0$, define
\(
\tau_\beta(t,s) \coloneqq 4.5\,\nu^2\sqrt{\beta\log(4t)/s}
\) and 
\(
E_g(t,s) \coloneqq \bigl(1 + |\hat Z_g(s)|/\sqrt{s}\bigr)\,\tau_\beta(t,s),
\)
and
\[
    e^+_g(t,s) \coloneqq
    \begin{cases}
        \tau_\beta(t,s)\,(1 + |\mu_g/\sigma_g|)/\bigl(1 - \tau_\beta(t,s)\bigr) & \text{if } \tau_\beta(t,s) < 1, \\
        +\infty & \text{otherwise.}
    \end{cases}
\]
Then there exists a constant $C_\nu > 0$ depending only on $\nu$ such that
\begin{equation}
    \mathbb{P}\left( \left|\frac{\mu_g}{\sigma_g} - \frac{\hat Z_g(s)}{\sqrt{s}} \right| < E_g(t,s) \le e^+_g(t,s) \right) \ge 1 -4t^{-\beta}. \label{eq:main-cb}
\end{equation}
\end{proposition}
We now establish a deviation bound that underlies Proposition~\ref{prop:self-normalized-cb}. Let $R_g(s) \coloneqq \hat Z_g(s)/\sqrt s$ denote the empirical signal-to-noise estimator at sample size $s$, so that $R_g(s) \to z_g \coloneqq \mu_g/\sigma_g$ in probability as $s \to \infty$.
\begin{lemma}[Unified self-normalized deviation bound]\label{lemma:unified-cb}
    Suppose Assumption~\ref{assn:sub-g} holds, so that
    $\|X_g(t)-\mu_g\|_{\psi_2}\le\nu\sigma_g$. For any $\delta>0$,
    \begin{equation}
        \mathbb{P}\!\left(
            \bigl|R_g(s)-z_g\bigr|
            >
            \delta\bigl(1+|R_g(s)|\bigr)
        \right)
        \;\le\;
        4\exp\!\left(
            -\frac{s\delta^2}{4.5^2\nu^4}
        \right).
        \label{eq:unified-cb-prob}
    \end{equation}
    Moreover, whenever $\delta<1$, on the same high-probability event the
    random-form bound in \eqref{eq:unified-cb-prob} implies
    \begin{equation}
        \bigl|R_g(s)-z_g\bigr|
        \;\le\;
        \frac{\delta(1+|z_g|)}{1-\delta}.
        \label{eq:unified-cb-deterministic}
    \end{equation}
\end{lemma}

\begin{proof}
Let $c_0$ denote a constant to be optimized in the proof, and set $x = s\delta^2/(c_0^2\nu^4)$. Put $Y_i =(X_g(i) - \mu_g)/\sigma_g$, and define
\[m_s = \frac{1}{s}\sum_{i=1}^s Y_i, \qquad Q_s = \frac{1}{s}\sum_{i=1}^s Y_i^2, \qquad \hat v_s = Q_s - m_s^2.\] Note that $R_g(s) = (m_s + z_g) / \sqrt{\hat v_s}$.

\paragraph{Step 1 (sample mean).} Put $Y_i =(X_g(s) - \mu_g)/\sigma_g$.
By \eqref{eq:concentration-standardized-mean}, outside an event of probability
$2e^{-x}$,
\[
    \left|m_s\right| \le \nu\sqrt{2x/s} = \frac{\delta}{\nu} \sqrt{2/c_0^2}.
\]
Since $\nu \ge 1$, this is at most $\delta$ whenever $c_0 \ge \sqrt{2}$.

\paragraph{Step 2 (lower tail of the sample variance).}
By \eqref{eq:concentration-standardized-cov-lb}, outside an event of probability at most
$e^{-x}$,
\[
    1-Q_s \le \frac{x}{3s} + \sqrt{ \frac{2\kappa(\nu)x}{s}+\frac{x^2}{9s^2}
    }.
\]
Consequently, after intersecting with the event from Step 1,
\begin{equation}
 1-\hat v_s \le \frac{x}{3s} + \sqrt{ \frac{2\kappa(\nu)x}{s}+\frac{x^2}{9s^2}
    } + \frac{2\nu^2x}{s}
    \label{eq:unified-cb-variance-lb}
\end{equation}
If $\delta\ge1$, then
\(
    1-\sqrt{\widehat v_s} \le
    1 \le  \delta
\)
deterministically. Suppose instead that $\delta<1$.
If we choose $x = s\delta^2/(c_0^2\nu^4)$, then \eqref{eq:unified-cb-variance-lb} gives
\begin{align*}
    1-\widehat v_s
    &\le
    \frac{x}{3s} + \sqrt{ \frac{2\kappa(\nu)x}{s}+\frac{x^2}{9s^2}
    } + \frac{2\nu^2x}{s} \\
    &= \delta \left( \frac{\delta}{3c_0^2\nu^4} + \sqrt{\frac{2\kappa(\nu)}{c_0^2 \nu^4} + \frac{\delta^2}{9c_0^4\nu^8}} + \frac{2\delta}{c_0^2 \nu^2}\right). \\
    &\le  \delta \left( \frac{1}{3c_0^2} + \sqrt{\frac{13.6}{c_0^2} + \frac{1}{9c_0^4}} + \frac{2}{c_0^2}\right),
\end{align*}
where in the last step we used $\nu \ge 1$, $\kappa(\nu) \le 6.8\nu^4$, and $\delta < 1$. Evidently, when $c_0 \ge 4.5$,
\begin{equation}
    1-\sqrt{\widehat v_s}
    \le
    \delta.
    \label{eq:unified-cb-sd-lb}
\end{equation}

\paragraph{Step 3 (upper tail of the sample variance).}
By \eqref{eq:concentration-standardized-cov-ub}, outside an event of probability
$e^{-x}$,
\[
    Q_s-1 \le \nu^2\left(\sqrt{\frac{32 x}{s} } + \frac{2x}{s}\right)
    =
   \left(\sqrt{\frac{32\delta^2}{c_0^2} } + \frac{2\delta^2}{c_0^2\nu^2}\right),
\]
where the second equality plugs in our choice of $x$. Since $\widehat v_s\le Q_s$ and $\nu\ge1$,
\[
    \widehat v_s-1  \le \sqrt{32}(\delta/c_0)+2(\delta/c_0)^2 \le 8(\delta/c_0)+16(\delta/c_0)^2
    = [1+ (4\,\delta)/c_0)]^2-1.
\]
Applying the monotone transformation $u \mapsto \sqrt{u + 1} - 1$ to both sides gives $\sqrt{\hat v_s} - 1 \le 4\,\delta / c_0$. 
Combining with \eqref{eq:unified-cb-sd-lb} gives that when $c_0 \ge 4.5$,
\begin{equation}
    \bigl|1-\sqrt{\widehat v_s}\bigr|
    \le
    \delta.
    \label{eq:unified-cb-sd}
\end{equation}

\paragraph{Step 4 (self-normalized deviation).}
Since
\(
    R_g(s)
    =
    {(z_g+m_s)}/{\sqrt{\widehat v_s}},
\)
we have the exact identity
\[
    R_g(s)-z_g
    =
    m_s
    +
    R_g(s)\bigl(1-\sqrt{\widehat v_s}\bigr).
\]
On the intersection of the events in Steps~1--3,
\[
    \bigl|R_g(s)-z_g\bigr|
    \le
    |m_s|
    +
    |R_g(s)|
    \bigl|1-\sqrt{\widehat v_s}\bigr|
    \le
    \delta\bigl(1+|R_g(s)|\bigr).
\]
A union bound gives total failure probability at most
$2e^{-x}+e^{-x}+e^{-x}=4e^{-x}$, which proves
\eqref{eq:unified-cb-prob}. Finally, on the same event,
\[
    \delta\bigl(1+|R_g(s)|\bigr)
    \le
    \delta\bigl(
        1+|z_g|+|R_g(s)-z_g|
    \bigr).
\]
If $\delta<1$, rearranging yields
\[
    |R_g(s)-z_g|
    \le
    \frac{\delta(1+|z_g|)}{1-\delta},
\]
which is \eqref{eq:unified-cb-deterministic}.
\end{proof}

\begin{proof}[Proof of Proposition \ref{prop:self-normalized-cb}]
Apply Lemma~\ref{lemma:unified-cb} with $\delta = \tau_\beta(t,s)$. When $\tau_\beta(t,s) <1$, the deterministic-form bound \eqref{eq:unified-cb-deterministic} reads
\[
    \left|R_g(s) - z_g\right| \;\le\; \frac{\tau_\beta(t,s)\,(1 + |z_g|)}{1 - \tau_\beta(t,s)} \;=\; e^+_g(t,s),
\]
and the random-form bound reads $|R_g(s) - z_g| < (1 + |R_g(s)|)\tau_\beta(t,s) = E_g(t,s)$. With $\delta = \tau_\beta(t,s) = 4.5\,\nu^2\sqrt{\beta\log(t)/s}$, the deviation probability is at most $4t^{-\beta}$. This establishes \eqref{eq:main-cb}.
\end{proof}

\subsection{Abstract regret bound}
Finally, we prove a high-probability regret bound for the SN-UCB algorithm. In particular the two-sided confidence bound of Proposition \ref{prop:self-normalized-cb} gives us: 
\begin{itemize}
    \item an estimate $\hat z_g(s) \coloneqq \hat Z_g(s)/\sqrt{s}$ of $\mu_g/\sigma_g$, where $s$ is the number of times arm $g$ has been drawn;
    \item a random, feasible ``exploration function'' $E_g(t,s)$ for $s \le t$ such that $\hat z_g(s) \pm E_g(t,s)$ likely contains $\mu_g/\sigma_g$; and
    \item a not-necessarily observed, deterministic upper bound $e^+_g(t,s)$ for $E_g(t,s)$.
\end{itemize}  
In this setting, we have the following result.
\begin{proposition}[Abstract regret bound for SN-UCB]
\label{prop:agnostic-ucb}
Suppose for each arm $g$ and each $s \le t$,
\begin{equation}
\mathbb{P}\left\{|\mu_g/\sigma_g - \hat z_g(s)| < E_g(t,s) \le e^+_g(t,s)\right\}
 \ge 1 - \psi(t). \label{eq:abstract-cb}
 \end{equation}
Suppose also that $\psi(-), e^+_g(t,-)$ are decreasing while $E_g(-,s)$ and $e^+_g(-,s)$ are increasing.
Consider the upper confidence bound policy which at time $t$ chooses the arm $g$ that maximizes \[\hat z_g(N_g(t-1)) + E_g(t,N_g(t-1)).\]
Define 
\(
\Delta_g = (\max_{g'} z_{g'}) - z_g,
\) and
\(
u^*_g = \inf\Set[s \ge 1]{2\,e^+_g(T,s) < \D_g}.
\)
For any sub-optimal $g$,
\[
\mathbb{P}(N_g(T) > u^*_g + q) \le 2\int_{u^*_g+q}^T \psi(t)\,dt
.\]
\end{proposition}

\begin{proof}[Proof of Proposition \ref{prop:agnostic-ucb}]
Without loss of generality, assume arm $1$ is optimal. Suppose that at time $t$ we choose arm $g>1$. Then one of the following
events must have taken place:
\begin{align*}
   \mathcal{E}_{1,t}
   &=
   \set{
       \hat z_1(N_1(t-1))
       +
       E_1(t,N_1(t-1))
       \le z_1
   }, \\
   \mathcal{E}_{2,t}
   &=
   \set{
       \hat z_g(N_g(t-1))
       -
       E_g(t,N_g(t-1))
       > z_g
   }
   \cup
   \set{
       E_g(t,N_g(t-1))
       >
       e^+_g(t,N_g(t-1))
   }, \\
   \mathcal{E}_{3,t}
   &=
   \set{
       \D_g
       <
       2\,e^+_g(T,N_g(t-1))
   }.
\end{align*}
Indeed, on
$(\mathcal E_{1,t}\cup\mathcal E_{2,t}\cup\mathcal E_{3,t})^c$,
we have
\begin{align*}
\hat z_1(N_1(t-1)) + E_1(t,N_1(t-1))
&> z_1
&&(\t{on }\mathcal E_{1,t}^c) \\
&= z_g+\D_g \\
&\ge z_g+2\,e^+_g(T,N_g(t-1))
&&(\t{on }\mathcal E_{3,t}^c) \\
&\ge z_g+2\,e^+_g(t,N_g(t-1))
&&(\t{by monotonicity}) \\
&\ge z_g+e^+_g(t,N_g(t-1))
       +E_g(t,N_g(t-1))
&&(\t{on }\mathcal E_{2,t}^c) \\
&\ge \hat z_g(N_g(t-1))
       +E_g(t,N_g(t-1))
&&(\t{on }\mathcal E_{2,t}^c),
\end{align*}
which implies that arm $g$ was not chosen.

Now put
\(
u
=
\inf\Set[s\ge1]{
    2\,e^+_g(T,s) < z_1-z_g
},
\)
and consider the event that $N_g(T)>u+q$, for a suboptimal arm $g>1$.
Then there must be a time $t_0$ satisfying
$u+q<t_0\le T$ at which $g$ is drawn for the
$(u+q+1)^{\mathrm{th}}$ time. At time $t_0$, arm $g$ has been drawn
exactly $u+q$ times, while arm $1$ has been drawn
$s\le t_0-u-q$ times. Correspondingly, define
\[
\mathcal E(t_0)
=
\set{
    N_g(t_0)>N_g(t_0-1)=u+q,
    \t{ and }
    N_1(t_0-1)\le t_0-u-q
}.
\]
Since $N_g(t_0-1)=u+q\ge u$ and $e_g^+(T,-)$ is decreasing,
we have
\[
    2e_g^+(T,N_g(t_0-1))
    \le
    2e_g^+(T,u)
    \le
    \D_g,
\]
and hence
$\mathcal E(t_0)\subset\mathcal E_{3,t_0}^c$.
Since a suboptimal arm is drawn, it follows that
\begin{align*}
\mathcal E(t_0)
&\subset
   \bigl(\mathcal E(t_0)\cap\mathcal E_{1,t_0}\bigr)
   \cup
   \bigl(\mathcal E(t_0)\cap\mathcal E_{2,t_0}\bigr) \\
&\subset
   \set{
       \exists s\le t_0-u-q
       \t{ s.t. }
       \hat z_1(s)+E_1(t_0,s)\le z_1
   } \\
&\qquad\cup
   \set{
       \hat z_g(u+q)-E_g(t_0,u+q)>z_g
   } \\
&\qquad\cup
   \set{
       E_g(t_0,u+q)>e_g^+(t_0,u+q)
   }.
\intertext{
Since $s+u+q\le t_0$ and $E_1(-,s)$ is increasing, we may expand
the first event to obtain
}
\mathcal E(t_0)
&\subset
   \set{
       \exists s\le t_0-u-q
       \t{ s.t. }
       \hat z_1(s)+E_1(s+u+q,s)\le z_1
   } \\
&\qquad\cup
   \set{
       \hat z_g(u+q)-E_g(t_0,u+q)>z_g
   } \\
&\qquad\cup
   \set{
       E_g(t_0,u+q)>e_g^+(t_0,u+q)
   }.
\end{align*}
It follows that
\begin{align*}
\set{N_g(T)>u+q}
&\subset
\bigcup_{t_0>u+q}\mathcal E(t_0) \\
&\subset
\left(
\bigcup_{1\le s\le T-u-q}
\set{
    \hat z_1(s)+E_1(s+u+q,s)\le z_1
}
\right. \\
&\qquad\quad
\cup
\bigcup_{u+q<t\le T}
\set{
    \hat z_g(u+q)-E_g(t,u+q)>z_g
} \\
&\left.\qquad\quad
\cup
\bigcup_{u+q<t\le T}
\set{
    E_g(t,u+q)>e_g^+(t,u+q)
}
\right).
\intertext{
Re-indexing the first union with $t=s+u+q$ gives
}
\set{N_g(T)>u+q}
&\subset
\bigcup_{u+q<t\le T}
\set{
    \hat z_1(t-u-q)+E_1(t,t-u-q)\le z_1
} \\
&\qquad\quad
\cup
\bigcup_{u+q<t\le T}
\left[
\set{
    \hat z_g(u+q)-E_g(t,u+q)>z_g
}
\right. \\
&\left.\qquad\qquad\qquad\qquad
\cup
\set{
    E_g(t,u+q)>e_g^+(t,u+q)
}
\right].
\end{align*}
The first event $\{\hat z_1 + E_1 \le z_1\} = \{z_1 - \hat z_1 \ge E_1\}$ is contained in $\{|z_1 - \hat z_1| \ge E_1\}$, and hence in the complement of the (strict) confidence event
\eqref{eq:abstract-cb} for arm $1$. The union of the second and third
events is contained in the complement of \eqref{eq:abstract-cb} for arm
$g$, since
\[
\set{\hat z_g-E_g>z_g}
\cup
\set{E_g>e_g^+}
\subset
\set{
    |z_g-\hat z_g| < E_g\le e_g^+
}^c.
\]
We may therefore apply a union bound to obtain
\begin{align*}
\mathbb P\{N_g(T)>u+q\}
&\le
2\sum_{t=u+q+1}^T\psi(t).
\end{align*}
Bounding the sum by an integral over $t\in[u+q,T]$ gives the result.
\end{proof}

\subsection{Proof of Theorem \ref{thm:sn-ucb-regret-bound}}\label{sec:proof-regret-thm}
Without loss of generality, assume the optimal arm has index $g = 1$. To apply Proposition \ref{prop:agnostic-ucb} we verify the hypothesis \eqref{eq:abstract-cb}, which we restate:
\[
    \mathbb{P}\!\left\{|\mu_g/\sigma_g - \hat z_g(s)| < E_g(t,s) \le e^+_g(t,s)\right\} \ge 1 - \psi(t).
\]
By Proposition \ref{prop:self-normalized-cb}, this hypothesis is satisfied with $\psi(t) = C_\nu t^{-\beta}$,
\[
    \tau_\beta(t,s) = 4.5\,\nu^2\sqrt{\beta\log(t)/s},
    \qquad
    E_g(t,s) = \bigl(1 + |\hat z_g(s)|\bigr)\,\tau_\beta(t,s),
\]
and the deterministic upper bound
\[
    e^+_g(t,s) =
    \begin{cases}
        \tau_\beta(t,s)\,(1 + |z_g|)/\bigl(1 - \tau_\beta(t,s)\bigr) & \text{if } \tau_\beta(t,s) < 1, \\
        +\infty & \text{otherwise,}
    \end{cases}
\]
where $z_g \coloneqq \mu_g/\sigma_g$. Recall from Proposition \ref{prop:agnostic-ucb} that $u^*_g = \inf\{s \ge 1 : 2\,e^+_g(T,s) < \Delta_g\}$. Writing $n_0(T) \coloneqq 22\nu^4\beta\log(T)$, the standard calculation gives
\[
    u^*_g \;\le\; n_0(T) \;+\; \frac{C_\nu(1 + |z_g|)^2\,\beta\log T}{\Delta_g^2}.
\]
Proposition \ref{prop:agnostic-ucb} then implies that for $g > 1$,
\[
    \mathbb{P}\bigl(N_g(T) > u^*_g + q\bigr) \;\le\; C \int_{u^*_g+q}^T t^{-\beta}\,dt.
\]
For $\beta > 2$, integrating the tail in $q$ yields $\mathbb{E}\,N_g(T) \lesssim u^*_g$. 
Summing over suboptimal arms with the per-arm bound on $u^*_g$,
\begin{align*}
    \mathbb{E}[R_T] = \mathbb{E}\!\left[\sum_{g: z_g < z^*} \Delta_g\,N_g(T)\right] 
    &\lesssim \sum_{g: z_g < z^*}^k \Delta_g\!\left[n_0(T) \;+\; \frac{C_\nu(1+|z_g|)^2\,\beta\log T}{\Delta_g^2}\right] \\
    &= \sum_{g: z_g < z^*}\!\left[\Delta_g\,n_0(T) \;+\; \frac{C_\nu(1+|z_g|)^2\,\beta\log T}{\Delta_g}\right].
\end{align*}
Since $n_0(T) = 22\nu^4\,\beta\log(T) \le C_\nu\,\beta\log T$, the warm-up contribution from arm $g$ is $\Delta_g\,n_0(T) \le C_\nu\,\beta\log T \cdot \Delta_g$, so
\[
    \mathbb{E}[R_T] \;\le\; C_\nu\,\beta\log T \sum_{g: z_g < z^*}\!\left\{\Delta_g + \frac{(1+|z_g|)^2}{\Delta_g}\right\}.
\]
The same calculation without the leading $\Delta_g$ in the sum gives
\begin{align*}
    \mathbb{E}[E_T] = \mathbb{E}\left[\sum_{g: z_g < z^*} N_g(T)\right] &\lesssim \sum_{g: z_g < z^*}\!\left[n_0(T) + \frac{C_\nu(1+|z_g|)^2\,\beta\log T}{\Delta_g^2}\right] \\ 
    &\le C_\nu\,\beta\log T \sum_{g: z_g < z^*}\!\left\{1 \;+\; \frac{(1+|z_g|)^2}{\Delta_g^2}\right\}. \qed
\end{align*}

\section{UCT Application: Calibration and Supporting Results}

\label{app:uct}

This appendix gives the calibration details and supporting simulations for the unconditional cash transfer (UCT) simulation of Section~\ref{sec:uct}.

\subsection{Setup}

A program variant (``arm'') is the combination of a \emph{targeting rule}, which determines which households are eligible for the transfer, and a transfer amount $T_g \in \{\$20, \$50\}$ per month. We consider seven targeting rules, defined in Section~\ref{sec:uct-rules} below, paired with both transfer amounts, for $k = 14$ arms in total. Outcomes are modeled as $X_g(t) \sim \mathcal{N}(\mu_g, \sigma_g^2)$, where $X_g(t)$ is the monthly consumption gain (USD) of the $t$-th household assigned to arm $g$. A variant clears the cost-effectiveness bar if it generates at least $c = 0.60$ dollars of consumption gain per dollar transferred, so the per-arm threshold is $u_g = c\, T_g$ (\$12 for the \$20 transfer, \$30 for the \$50 transfer). Working with shifted observations $\tilde X_g(t) = X_g(t) - u_g$ reduces the problem to the standard form $\mathcal{H}_0: \max_g \tilde\mu_g \le 0$ used throughout the paper.

\subsection{Targeting Rules}
\label{sec:uct-rules}

Each rule is described in one sentence; in each case the rule determines the eligible population, and every household in the eligible population receives the transfer.
\begin{itemize}
  \item \emph{Universal:} every household in the program area.
  \item \emph{Geographic (rural):} every household in a designated rural area.
  \item \emph{Proxy means test (PMT):} households scoring in the bottom 20\% of an asset-based poverty index (e.g., flooring material, livestock, durable goods).
  \item \emph{Demographic:} households containing at least one child under age five.
  \item \emph{Community-based targeting (CBT):} households selected by community members, typically through a village ranking exercise.
  \item \emph{Self-targeting:} households that complete a costly application (e.g., travel to a central enrollment site), which screens out better-off households that find the application not worth the effort.
  \item \emph{Categorical (elderly):} households containing at least one member above an age cutoff (typically 65 or 70).
\end{itemize}
The first four rules above are widely studied in the program-evaluation literature; we treat them as core. Community-based targeting, self-targeting, and the elderly categorical rule are included to broaden the menu and are described further in Section~\ref{sec:uct-rules-extra}.

\subsection{Calibrated Arm Means}

For each arm we write $\mu_g = \rho_g \cdot T_g$, where $\rho_g \in [0,1]$ is the \emph{gain rate}: the fraction of each transferred dollar that shows up as measured monthly consumption gain in the eligible population. The picks for $\rho_g$ are anchored in published evaluations of comparable programs; we describe the picks for the core rules here and the additional rules in Section~\ref{sec:uct-rules-extra}. Two examples illustrate the approach:
\begin{itemize}
  \item \emph{Geographic, \$50 transfer.} \citet{haushofer2016short} report near-unit pass-through from rural transfer to measured consumption in GiveDirectly's Kenya program. We use $\rho = 0.80$ -- a downward adjustment to reflect the broader ``rural'' rule -- giving $\mu = \$40$.
  \item \emph{PMT, \$20 transfer.} \citet{merttens2013kenya} report sizable food-consumption gains among PMT-selected households in the Kenya Hunger Safety Net Programme. We use $\rho = 0.80$, giving $\mu = \$16$.
\end{itemize}
Universal transfers reach many households well above the poverty line and translate a smaller share of each dollar into measured consumption; we use $\rho \approx 0.40$, broadly consistent with the cross-program review of \citet{bastagli2016cash}. Demographic rules sit between universal and PMT in selectivity; we use $\rho \approx 0.65$ \citep{aguero2009impact, kabeer2012economic}.

\subsection{Calibrated Variances}

The standard deviations $\sigma_g$ reflect two sources of household-level spread.
\begin{enumerate}
  \item \emph{Heterogeneity of the eligible population.} Universal rules reach the full income distribution and so produce highly variable responses; PMT-selected households are more homogeneous.
  \item \emph{Composition of consumption.} Small transfers to very poor households go mostly to predictable necessities (food, medicine), with low household-to-household variance \citep{dupas2013savings}. Larger transfers, especially under broader targeting rules, are often used for one-off investments such as livestock, business equipment, or home repair, whose short-run consumption pass-through is highly variable \citep{blattman2014generating, egger2022general}.
\end{enumerate}
The two effects compound: within each rule, both $\mu_g$ and $\sigma_g$ rise with the transfer amount, and at a fixed transfer $\sigma_g$ tends to be larger for broader rules. The specific magnitudes of $\sigma_g$ are chosen in a stylized manner reflecting these two forces.

The resulting parameter table is given in Table~\ref{tab:uct-arms}. The structural feature relevant for adaptive sampling is that the arm with the largest \emph{net} mean (Geographic-High, $\tilde\mu = \$10$) is not the arm with the largest signal-to-noise ratio (PMT-Low, SNR $= 0.50$).

\begin{table}[htbp]
\centering
\small
\begin{tabular}{clcccccc}
\toprule
Arm & Targeting & $T_g$ & $\mu_g$ & $\sigma_g$ & $u_g$ & $\tilde\mu_g$ & SNR \\
\midrule
1  & Universal     & \$20 & \$8  & 20 & \$12 & $-$\$4  & null \\
2  & Universal     & \$50 & \$20 & 45 & \$30 & $-$\$10 & null \\
3  & Geographic    & \$20 & \$16 & 16 & \$12 & $+$\$4  & 0.25 \\
4  & Geographic    & \$50 & \$40 & 45 & \$30 & $+$\$10 & 0.22 \\
5  & PMT           & \$20 & \$16 &  8 & \$12 & $+$\$4  & \textbf{0.50} \\
6  & PMT           & \$50 & \$36 & 22 & \$30 & $+$\$6  & 0.27 \\
7  & Demographic   & \$20 & \$13 & 14 & \$12 & $+$\$1  & 0.07 \\
8  & Demographic   & \$50 & \$32 & 35 & \$30 & $+$\$2  & 0.06 \\
\midrule
9  & CBT           & \$20 & \$13 & 15 & \$12 & $+$\$1  & 0.07 \\
10 & CBT           & \$50 & \$33 & 32 & \$30 & $+$\$3  & 0.09 \\
11 & Self-target.  & \$20 & \$14 & 18 & \$12 & $+$\$2  & 0.11 \\
12 & Self-target.  & \$50 & \$33 & 40 & \$30 & $+$\$3  & 0.08 \\
13 & Elderly       & \$20 & \$13 & 20 & \$12 & $+$\$1  & 0.05 \\
14 & Elderly       & \$50 & \$33 & 42 & \$30 & $+$\$3  & 0.07 \\
\bottomrule
\end{tabular}
\caption{UCT program variants (arms). $T_g$ is the monthly transfer; $\mu_g$ and $\sigma_g$ are the gross consumption-gain mean and standard deviation; $u_g = 0.60 \cdot T_g$ is the cost-effectiveness threshold; $\tilde\mu_g = \mu_g - u_g$ is the net mean; SNR is $\tilde\mu_g/\sigma_g$. Arms 1--2 are null. The largest-mean arm (4, Geographic-High) and the largest-SNR arm (5, PMT-Low) differ. Arms 9--14 are additional targeting rules proposed in the literature; their parameter values and justification are described in Appendix~\ref{sec:uct-rules-extra}.}
\label{tab:uct-arms}
\end{table}

\subsection{Algorithms and Allocation Paths}

We compare SN-UCB (Algorithm~\ref{algo:sn-ucb}), three mean-based adaptive baselines, and uniform allocation paired with a Bonferroni-corrected one-sided $t$-test. The mean-based baselines are: standard UCB \citep{auer2002finite}; a variance-aware UCB \citep{audibert2009exploration} with bonus $\hat\sigma_g \sqrt{c\, \log t / N_g(t)}$; and Thompson sampling \citep{thompson1933likelihood}. 

Figure~\ref{fig:uct-allocation} traces the average allocation $N_g(t)/t$ for each arm under SN-UCB and the variance-aware UCB at $T = 500$, averaged over 1{,}000 replications. SN-UCB devotes the largest share of its samples to PMT-Low, with PMT-High and the geographic arms also receiving meaningful allocations and the remaining rules collectively absorbing a smaller share. Variance-aware UCB instead concentrates the bulk of its samples on Geographic-High. Both rules allocate few samples to the Universal arms, which carry strongly negative net means. The allocation paths explain the power gap in Figure~\ref{fig:uct-power}: SN-UCB concentrates on the best-SNR arm, which generates more statistical evidence per sample than UCB's preferred largest-mean arm.

\begin{figure}[htbp]
\centering
\includegraphics[width=0.95\textwidth]{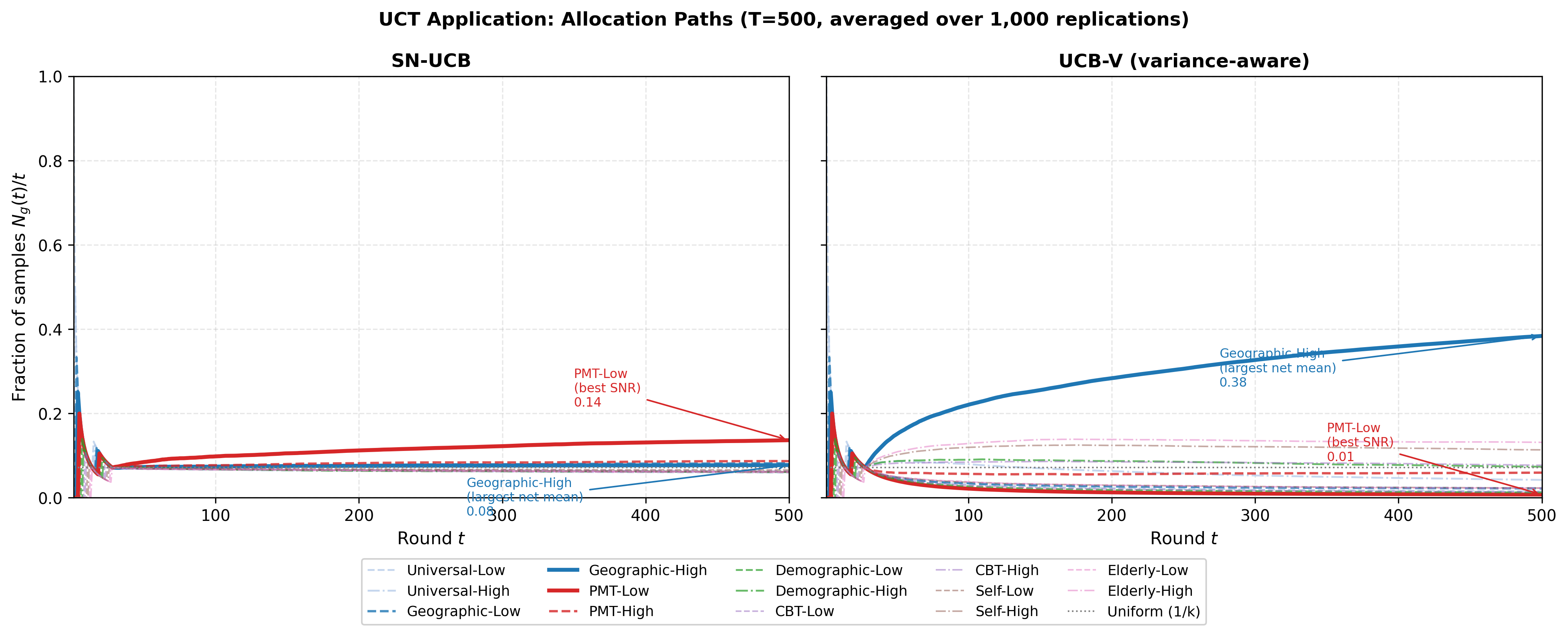}
\caption{Average sample-allocation fractions $N_g(t)/t$ for each of the fourteen UCT arms over a horizon of $T = 500$, averaged over 1{,}000 Monte Carlo replications. Left: SN-UCB concentrates on PMT-Low (arm~5), the highest-SNR arm. Right: a variance-aware UCB benchmark concentrates on Geographic-High (arm~4), the highest-mean arm. Both rules deallocate away from the Universal arms (1 and 2), which fall below the cost-effectiveness threshold.}
\label{fig:uct-allocation}
\end{figure}

\subsection{Additional Targeting Rules}
\label{sec:uct-rules-extra}

The empirical anchoring for community-based targeting, self-targeting, and the elderly categorical rule (arms 9--14) is looser than for the four core rules. In each case we cite one or two studies that motivated the parameter pick, but the gain rates and standard deviations are not estimates.

\begin{itemize}
  \item \emph{Community-based targeting (arms 9--10).} Gain rate $\rho = 0.65$; $\sigma = 15$ at the low transfer and $32$ at the high transfer. \citet{alatas2012targeting} compare PMT, CBT, and a hybrid in 640 Indonesian villages; the paper reports targeting accuracy rather than consumption pass-through, and we infer a somewhat-below-PMT gain rate from the modest CBT-versus-PMT inclusion-error gap. \citet{sumarto2024community} document broadly similar performance at scale.
  \item \emph{Self-targeting (arms 11--12).} Gain rate $\rho = 0.70$ at the low transfer and $0.66$ at the high transfer; $\sigma = 18$ and $40$. \citet{alatas2016self} run a village-level experiment in Indonesia in which beneficiaries must travel to a central site to apply, in place of automatic enrollment based on the same asset test; the application requirement selects substantially poorer applicants than direct enrollment. We pick a gain rate close to PMT among the self-selected applicants and attenuate slightly at the higher transfer level.
  \item \emph{Categorical, elderly (arms 13--14).} Gain rate $\rho = 0.65$; $\sigma = 20$ and $42$. \citet{bando2020non-contributory} report 71\% of Mexico's 70 y M{\'a}s pension flowing to shared household consumption, with corresponding household-consumption gains of roughly 23\%; non-contributory pension programs in Peru and Paraguay produce gains in the 40\%--44\% range. The high $\sigma$ values reflect heterogeneity across recipient households documented by \citet{duflo2003grandmothers, ardington2009labor, case1998large}.
\end{itemize}


\spacingset{1.25}

\end{document}